\documentclass[10pt]{amsart}
\usepackage[margin=1in,letterpaper,portrait]{geometry}
\usepackage{latexsym,amssymb,verbatim,multirow,graphicx,epsfig,enumerate, paralist,xypic}
\usepackage{xcolor}
\xyoption{curve}

\theoremstyle{plain}
   \newtheorem{theorem}{Theorem}[section]
   \newtheorem{proposition}[theorem]{Proposition}
   \newtheorem{lemma}[theorem]{Lemma}
   \newtheorem{corollary}[theorem]{Corollary}
   \newtheorem{conjecture}[theorem]{Conjecture}
   
   \newtheorem*{theorem*}{Theorem}
\theoremstyle{definition}
   \newtheorem{definition}[theorem]{Definition}
   \newtheorem{example}[theorem]{Example}
   
   \newtheorem{question}[theorem]{Question}
   \newtheorem{remark}[theorem]{Remark}

\numberwithin{equation}{section}

\newcommand\Symm{\mathfrak{S}}

\newcommand\xx{{\mathbf{x}}}

\newcommand{\boldt}{{\bf t}}
\newcommand{\boldalpha}{{\boldsymbol{\alpha}}}

\newcommand\wt{\operatorname{wt}}

\newcommand{\defn}[1]{{\it {#1}}}
\newcommand{\type}[1]{\mathrm{#1}}

\newcommand\CCC{{\mathbf{C}}}
\newcommand\ccc{{\mathbf{c}}}

\newcommand\CC{{\mathbb{C}}}
\newcommand\ZZ{{\mathbb{Z}}}
\newcommand\NN{{\mathbb{N}}}
\newcommand\QQ{{\mathbb{Q}}}
\newcommand\RR{{\mathbb{R}}}

\newcommand\ooo{{\mathfrak{o}}}

\newcommand{\TBD}[1]{\textbf{\color{red} #1}}

\title{Circuits and Hurwitz action in finite root systems}

\author{Joel Brewster Lewis}
\address{J.\ B. Lewis and V. Reiner, School of Mathematics, University of Minnesota, Minneapolis, MN 55455, USA}
\author{Victor Reiner}


\date{\today}

\begin{document}

\begin{abstract}
In a finite real reflection group, two factorizations of a Coxeter element 
into an arbitrary number of reflections are shown to lie in the same 
orbit under the Hurwitz 
action if and only if they use the same multiset of conjugacy classes.  
The proof makes use of a surprising lemma, derived from a classification of
the minimal linear dependences (matroid circuits) in finite root systems:
any set of roots forming a minimal linear dependence with
positive coefficients has a disconnected graph of 
pairwise acuteness.

\end{abstract}

\maketitle

%\tableofcontents

%%%%%%%%%%%%%%%%%%%%%%%%%%%%%%%%%%%%%%%%%%%%%%%%%%%%%%%%%%%%%
\section{Introduction}
\label{intro section}
%%%%%%%%%%%%%%%%%%%%%%%%%%%%%%%%%%%%%%%%%%%%%%%%%%%%%%%%%%%%%

Given a group $W$ and set $T$ of generators for $W$, consider factorizations $(t_1, t_2, \ldots, t_m)$ of a given element $g = t_1\cdots t_m$ in $W$.  When $T$ is closed under conjugation, these factorizations carry a natural action of the Artin braid group on $m$ strands called the \defn{Hurwitz action}.  Here the braid group generator $\sigma_i$ acts on ordered factorizations by a \defn{Hurwitz move}, interchanging two factors $t_i, t_{i+1}$
while conjugating one by the other:
\begin{equation}
\label{Hurwitz-move-definition}
\begin{array}{rccll}
(t_1,\ldots,t_{i-1},& t_i,&t_{i+1},& t_{i+2},\ldots,t_m) &\overset{\sigma_i}{\longmapsto} \\
(t_1,\ldots,t_{i-1},& t_{i+1},& t_i^{t_{i+1}},& t_{i+2},\ldots,t_m).
\end{array}
\end{equation}
(We use the notation $a^b := b^{-1}ab$ for conjugation in a group.)
When $W$ is a finite real reflection group of rank $n$ and $T$ is the
set of all of its reflections, 
%and taking $m=n$, 
D. Bessis used a simple inductive argument to prove the following result
about shortest factorizations of {\it Coxeter elements} 
(see Section~\ref{main theorem section}
for the definition), 
which he called the \emph{dual Matsumoto property}.

\vskip.1in
\noindent
{\bf Bessis's Theorem} (\cite[Prop.~1.6.1]{Bessis}){\bf.}
{\it 
%\label{bessis theorem}
Let $W$ be a finite real reflection group of rank $n$ and 
let $c$ be a Coxeter element of $W$.  
The set of all shortest ordered factorizations $(t_1,\ldots,t_n)$ of
$c=t_1 t_2 \cdots t_n$ as a product of reflections forms
a single transitive orbit under the Hurwitz action.  
}

\vskip.1in
The original context for this result is the \emph{dual Coxeter theory} developed by Bessis~\cite{Bessis} and Brady and Watt~\cite{Brady, BradyWattKPi1}.  
It has since been extended to several other contexts:
\begin{compactitem}
\item shortest reflection factorizations in \defn{well-generated 
complex reflection groups} \cite[Prop.~7.6]{BessisKpi1}, 
\item shortest \defn{primitive} factorizations in well-generated 
complex reflection groups \cite[Thm.~0.4]{Ripoll},
where primitivity means having at most one nonreflection factor,
\item shortest reflection factorizations in not-necessarily-finite 
Coxeter groups~\cite{BDSW}, and 
\item the classification in finite real reflection groups of
the elements whose shortest reflection factorizations
have a single Hurwitz orbit \cite{BGRW}.  
\end{compactitem}

However, the question of how Bessis's Theorem extends to \emph{longer} 
reflection factorizations seems not to have been addressed.  One obstruction
to transitivity has been noted frequently \cite{LandoZvonkin, LRS, Ripoll}:
the Hurwitz action preserves the 
(unordered) $m$-element multiset of conjugacy classes 
of the factors.  This multiset is 
called the \defn{unordered passport} in type $\type{A}$ by Lando and Zvonkin \cite[\S 5.4.2.2]{LandoZvonkin}.
In considering reflection factorizations of a Coxeter element $c$ 
whose length is strictly greater than the minimum (the rank $n$ of $W$), 
it is possible for the factorizations to use 
different multisets of reflection conjugacy classes.
When $W$ is a finite real reflection group, we show that this
is the only obstruction.

\begin{theorem}
\label{main theorem}
In a finite real reflection group, two reflection factorizations 
of a Coxeter element lie in the same Hurwitz orbit 
if and only if they share the same multiset of conjugacy classes.
\end{theorem}

\noindent
In particular, in the irreducible ``oddly-laced types'' 
($\type{A}_n, \type{D}_n, \type{E}_6, \type{E}_7, \type{E}_8,
\type{H}_3, \type{H}_4,$ and $\type{I}_2(m)$ with $m$ odd),
there is only one conjugacy class of reflections, and hence
the Hurwitz action is transitive.

We sketch here the proof of Theorem~\ref{main theorem}, 
which has three main steps.
The first is a lemma, proven in Section~\ref{section: acuteness},
that one might paraphrase as asserting that
``root circuits are acutely disconnected''.
Call a subset $C=\{\alpha_1,\ldots,\alpha_m\}$ of
a Euclidean space $(V,\langle \cdot, \cdot \rangle)$
a \defn{minimal dependence} (or \defn{circuit}) 
if there exist nonzero coefficients $c_i$ 
such that $\sum_{i=1}^m c_i \alpha_i=0$,
and $C$ is inclusion-minimal with respect to this property.
Define its \defn{acuteness graph} $\Gamma_C$ to have vertices
$\{1,2,\ldots,m\}$ and an edge $\{i,j\}$ whenever 
$\langle c_i \alpha_i, c_j \alpha_j\rangle>0$.
 
\begin{lemma}
\label{acuteness-lemma}
In a finite not-necessarily-crystallographic root system,
every circuit $C$ has $\Gamma_C$ disconnected.
\end{lemma}

\noindent
The second step (Section~\ref{main lemma section}) uses 
Lemma~\ref{acuteness-lemma} to prove a lemma on  
the \defn{absolute (reflection) length} function
$$
\ell_T(w):=\min\{\ell: w = t_1 t_2 \cdots t_\ell \text{ for some }t_i \in T \}.
$$

\begin{lemma}
\label{main lemma}
For any reflection factorization $\boldt = (t_1, \ldots, t_m)$ of
$w = t_1 \cdots t_m$ with $\ell_T(w) < m$,
either $m = 2$, or there exists $\boldt' = (t'_1, \ldots, t'_m)$ 
in the Hurwitz orbit of $\boldt$ with 
$\ell_T(t'_1 \cdots t'_k) < k$ for some $k \leq m-1$.
\end{lemma}

\noindent
The third step, also in Section~\ref{main lemma section},
iterates Lemma~\ref{main lemma} to put reflection factorizations 
into a standard form.

\begin{corollary}
\label{cor:standard form}
If $\ell_T(w)=\ell$, then every factorization of $w$ into $m$ reflections
lies in the Hurwitz orbit of some $\boldt=(t_1, \ldots, t_m)$ such that
\begin{align*}
t_1 &= t_2,\\
t_3 &= t_4,\\
    &\vdots \\
t_{m - \ell - 1} &= t_{m - \ell},
\end{align*} 
and
$(t_{m - \ell + 1}, \ldots, t_{m})$ is a shortest
reflection factorization of $w$.
\end{corollary}

\noindent
Section~\ref{main theorem section} then
finishes off the proof of Theorem~\ref{main theorem},
using the case of Corollary~\ref{cor:standard form} 
where $w$ is a Coxeter element $c$, along
with Bessis's Theorem above, and Bessis's observation that any reflection $t$ can occur 
first in a shortest factorization of $c$.
Section~\ref{remarks section} collects a few remarks
and questions suggested by this work.

We note that the proof of the crucial Lemma~\ref{acuteness-lemma} 
is case-based and 
relies on large computer calculations.  The remaining steps of the argument 
are case-independent (at least in the crystallographic case), so that one 
might hope to make the argument fully human-comprehensible by giving a 
case-free proof of Lemma~\ref{acuteness-lemma}.

\section*{Acknowledgements}
This work was partially supported by NSF grants DMS-1148634 and DMS-1401792.
The authors thank 
Guillaume Chapuy, 
Theodosios Douvropoulos, 
Vivien Ripoll, 
and Christian Stump for helpful conversations.
They also thank Patrick Wegener both for helpful comments,
and for the content of Section~\ref{Wegener-section}.

%%%%%%%%%%%%%%%%%%%%%%%%%%%%%%%%%%%%%%%%%%%%%%%%%%%%%%%%%%%%%
\section{Background and terminology}
\label{background section}
%%%%%%%%%%%%%%%%%%%%%%%%%%%%%%%%%%%%%%%%%%%%%%%%%%%%%%%%%%%%%

In this section, we review some standard definitions and facts about finite real reflection groups and root systems.  Good references for this material are \cite[Chs.~1, 4]{BjornerBrenti}, \cite{Humphreys}, and \cite[\S\S 2.1--2.2]{Armstrong}.

\begin{definition}
Let $(V, \langle \cdot, \cdot \rangle)$ be a finite-dimensional
Euclidean space, that is, a real vector space $V \cong \RR^n$ with a positive 
definite symmetric bilinear from $\langle \cdot, \cdot \rangle$,
whose associated norm $|v|$ is given by $|v|^2=\langle v, v \rangle$.
For a vector $\alpha$ in $V$, the \defn{reflection} $s_\alpha$ through the hyperplane $H=\alpha^\perp$ is the linear map given by the formula
\begin{equation}
\label{reflection formula}
s_\alpha(v) = v - \frac{2 \langle v,\alpha\rangle}
                       {|\alpha|^2} \alpha.
\end{equation}
A \defn{finite reflection group} is a 
finite subgroup $W$ of $GL_n(\RR)$ generated by
its subset $T \subset W$ of reflections.  
\end{definition}

Since reflections lie
within the orthogonal group $O_n(\RR)$, so does $W$.  That is,
$W$ preserves $\langle \cdot,\cdot \rangle$.  

\begin{definition}
A \defn{(finite, reduced, not-necessarily-crystallographic) root system}
associated to a finite reflection group $W$ is any $W$-stable subset $\Phi \subset V$ 
consisting of a choice of two opposite normal vectors 
$\pm \alpha$ for each reflecting
hyperplane $H$ of a reflection $t$ in $T$. 
We will assume $W$ has no fixed vector in $V$, that is, 
the $w$-fixed spaces defined by $V^w:=\{ v \in V: w(v)=v\}$ satisfy 
$
\bigcap_{w \in W} V^w=\{0\}.
$
\end{definition}

It is not hard to see that root systems $\Phi$ for $W$
are parametrized by picking a representative $t$ of each conjugacy class
of reflection and choosing a scaling for the normal vectors $\pm \alpha$ to the reflecting hyperplane of $t$.  On the other hand, one can axiomatize such root systems as follows:
they are the collections of finitely many nonzero vectors 
$\Phi \subset V$ with the property that
$s_\alpha(\beta) \in \Phi$ for all $\alpha, \beta \in V$,
and $\Phi \cap \RR \alpha =\{\pm \alpha \}$ for all $\alpha$ in $\Phi$.
In this case, one recovers $W$ as the group generated by the reflections 
$\{s_\alpha: \alpha \in \Phi\}$.

%%%%%%%%%%%%%%%%%%%%%%%%%%%%%%%%%%
% Maybe we never really need to discuss Coxeter systems?
\begin{comment}

One can classify finite reflection groups $W$ via their
\defn{Coxeter systems}.  The reflecting hyperplanes of elements
of $T$ decompose $V$ into simplicial cones, called \defn{Weyl chambers},
on which $W$ acts simply transitively.  After choosing a particular
Weyl chamber, the set of \defn{simple reflections} through
its walls, $S=\{s_1,\ldots,s_n\}$, 
turn out to generate $W$, giving it a \defn{Coxeter
presentation} $(W,S)$, that is, one of the form
\[
s_i^2 = e \qquad \textrm{ and } \qquad (s_i s_j)^{m_{i,j}} = e \; \textrm{ for } \; i \neq j
\text{ with }m_{i,j} \in \{2,3,\ldots\}
\]

\begin{definition}
The \defn{Coxeter diagram} associated with the Coxeter system $(W,S)$
has a node for each $s_i$ in $S$, and an edge labeled $m_{i,j}$ between
$s_i$ and $s_j$ whenever $m_{i,j} \geq 3$; there is no edge when
$m_{i,j}=2$ so that $s_i, s_j$ commute.  
\end{definition}

\end{comment}
%%%%%%%%%%%%%%%%%%%%%%%%%%%%%%%%%%%

\begin{definition}
An  \defn{open Weyl chamber} $F$ for a finite reflection group $W$ is
a connected component of the complement within $V$ of the
union of the reflecting hyperplanes for all reflections $t$ in $T$.
\end{definition}

\noindent
It turns out that $W$ acts simply transitively on the set of
Weyl chambers.  Also, the closure $\overline{F}$ of any Weyl chamber $F$ is a {\it fundamental
domain} for the action of $W$ on $V$:  every $W$-orbit $Wv$ on $V$
has $|(Wv) \cap \overline{F}|=1$.

\begin{definition}
The set $\Phi^+$ of \defn{positive roots} corresponding to a 
choice of an open Weyl chamber $F$ is 
\[
\Phi^+:=\{ \alpha \in \Phi: \langle \alpha, v \rangle > 0 
\text{ for all }v\text{ in }F\}.
\]
The associated set of \defn{simple roots} 
$\Pi \subset \Phi^+$ 
is the set of inward-pointing normal vectors to the walls of $\overline{F}$.
\end{definition}

It is easily seen that $\Phi = \Phi^+ \sqcup \left( -\Phi^+ \right)$.
Less obvious are the following properties of the simple roots 
$\Pi=\{\alpha_1,\ldots,\alpha_n\}$:
\begin{compactitem}
\item they are pairwise non-acute, 
\item they form an $\RR$-basis for $V$, 
\item they contain $W$-orbit representatives for all of the roots, and
\item every $\alpha \in \Phi^+$ has its 
unique expression $\alpha=\sum_{i=1}^n c_i \alpha_i$ with
$c_i \geq 0$ for all $i$.
\end{compactitem}

\begin{definition}
A finite reflection group $W$ is called \defn{reducible} if there
exists a nontrivial orthogonal direct sum decomposition $V=V_1 \oplus V_2$
respected by $W$.  
\end{definition}

\noindent
Reducibility of the group $W$ is equivalent to the
existence of a nontrivial decomposition 
$\Phi=\Phi_1 \sqcup \Phi_2$ with $\langle \alpha_1, \alpha_2 \rangle=0$
when $\alpha_i \in \Phi_i$ for $i=1,2$, in any (or every) 
root system $\Phi$ for $W$.
It is also equivalent to the existence of a
nontrivial decomposition 
$\Pi=\Pi_1 \sqcup \Pi_2$  with 
$\langle \alpha_1, \alpha_2 \rangle=0$
when $\alpha_i \in \Pi_i$ for $i=1,2$, in any (or every) 
choice of simple roots $\Pi$ for $\Phi$.
%It is also equivalent to a decomposition of $S=S_1 \sqcup S_2$ 
% where all $s$ in $S_1$ commute with $s'$ in $S_2$,
% or a nontrivial decomposition of the Coxeter diagram 
% into connected components.  
In this situation, $W=W_1 \times W_2$ where $W_i$ is the
subgroup generated by $\{s_\alpha: \alpha \in \Phi_i\}$, 
or by $\{s_\alpha: \alpha \in \Pi_i\}$.
% or generated by $S_i$.

There is a \defn{classification} of finite irreducible reflection groups $W$. 
It contains four infinite families and six exceptional groups: 
\begin{compactitem}
\item type $\type{A}_{n-1}$ for $n \geq 2$, where $W$ is isomorphic
to the symmetric group on $n$ letters,
\item type $\type{B}_n/\type{C}_n$ for $n \geq 2$, where $W$ is the hyperoctahedral
group of $n \times n$ signed permutation matrices,
\item type $\type{D}_n$ for $n \geq 4$, where $W$ is an index two
subgroup of the hyperoctahedral group,
\item type $\type{I}_2(m)$ for $m \geq 3$, where $W$ is the dihedral
group of symmetries of a regular $m$-gon, and
\item exceptional types $\type{E}_6, \type{E}_7, \type{E}_8, \type{F}_4, \type{H}_3, \type{H}_4$.
\end{compactitem}

We will later need to consider the field extension $K$ of $\QQ$
that adjoins to $\QQ$ the elements
$\frac{\langle \beta, \alpha\rangle}
{|\alpha|^2}$ for all $\alpha, \beta$ in $\Phi$.
If we normalize all of the roots to the same length, then 
$
\frac{\langle \beta, \alpha\rangle}{|\alpha|^2} = \cos\left(\frac{2\pi}{m}\right)
$
if the rotation $s_\alpha s_\beta$ has order $m$.
This number is always algebraic, so we may assume that $K$ is
a number field, that is, a finite extension of $\QQ$.
We can sometimes do better.

\begin{definition}
Say a root system $\Phi$ is \defn{crystallographic} if
\begin{equation}
\frac{2\langle \beta, \alpha\rangle}{|\alpha|^2} \in \ZZ
\text{ for all }
\alpha,\beta \in \Phi.
\end{equation}
\end{definition}

\noindent
Of course, if $\Phi$ is crystallographic, then $K=\QQ$.
Since rescaling roots within a $W$-orbit does not affect $W$ itself
(or any of the circuit properties to be discussed later), we 
always choose without further mention 
a crystallographic root system $\Phi$ for $W$ when one is available.
This means that $\Phi$ will be chosen crystallographic
in all types except $\type{H}_3$, $\type{H}_4$ 
(where one can take $K=\QQ[\sqrt{5}]$), and $\type{I}_2(m)$ for 
$m \not\in \{3, 4, 6\}$.

%%%%%%%%%%%%%%%%%%%%%%%%%%%%%%%%%%%%%%%%%%%%%%%%%%%%%%%%%%%%%
\section{Circuit classification and proof of Lemma~\ref{acuteness-lemma}}
\label{section: acuteness}
%%%%%%%%%%%%%%%%%%%%%%%%%%%%%%%%%%%%%%%%%%%%%%%%%%%%%%%%%%%%%

The goal of this section is to prove
Lemma~\ref{acuteness-lemma} from the Introduction, which
we recall here.  Fix a finite 
(not-necessarily-crystallographic) root system
$\Phi$ in a Euclidean space $(V,\langle \cdot,\cdot \rangle)$.

\begin{definition}
A finite subset $C=\{\alpha_1,\ldots,\alpha_m\} \subseteq V$ 
is a {\it circuit} if it has a nontrivial dependence $c_1 \alpha_1 + \cdots + c_m \alpha_m=0$, but no proper subset of $C$ is dependent.
Given a circuit $C$, the dependence coefficients $(c_1, \ldots, c_m) \in \RR^m$ are uniquely determined up to simultaneous $\RR$-scaling. 
Thus, one may define the \defn{acuteness graph} $\Gamma_C$ to have vertex set $\{1,2,\ldots,m\}$
and an edge $\{i,j\}$ whenever $\langle c_i \alpha_i, c_j \alpha_j \rangle > 0$.
\end{definition}

\vskip.1in
\noindent
{\bf Lemma~\ref{acuteness-lemma}.}
{\it 
In a finite root system,
every circuit $C$ has disconnected acuteness graph $\Gamma_C$.
}

\vskip.1in
We will often abuse terminology by considering two circuits $C, C'$ to be 
the {\it same} when they span the same set of lines 
$\{\RR \alpha \}_{\alpha \in C} = \{\RR \alpha' \}_{\alpha' \in C'}$,
or have the same set of normal hyperplanes 
$\{\alpha^\perp \}_{\alpha \in C} = \{(\alpha')^\perp \}_{\alpha' \in C'}$.
Note that in this case, $\Gamma_C=\Gamma_{C'}$.  In fact, our
figures below will depict slightly more graphical information about
the circuits $C$, namely an {\it acuteness-obtuseness graph} that shows
the $c_i \alpha_i$ labeling vertices, and these solid (acute) and 
dotted (obtuse) edges:
\[
\begin{cases}
\xymatrix{c_i \alpha_i\ar@{-}[r] &c_j \alpha_j}
&    \text{ when }\langle c_i \alpha_i, c_j \alpha_j \rangle >0, \\
\xymatrix{c_i \alpha_i\ar@{.}[r] &c_j \alpha_j}
&  \text{ when }\langle c_i \alpha_i, c_j \alpha_j \rangle <0,  \\
\xymatrix{c_i \alpha_i &c_j \alpha_j}
&  \text{ when }\langle c_i \alpha_i, c_j \alpha_j \rangle =0.
\end{cases}
\]
The acuteness graph $\Gamma_C$ comes from erasing the
dotted (obtuse) edges in the acuteness-obtuseness graph.

Our proof of Lemma~\ref{acuteness-lemma} 
relies on a classification of circuits in finite 
root systems, which may be of independent interest.  Such
a classfication is essentially already provided in the 
classical types $\type{A}_{n-1}, \type{B}_n/\type{C}_n, \type{D}_n$ by
Zaslavsky's theory of {\it signed graphs} \cite{Zaslavsky},
and we rely on a computer calculation for the exceptional types.

\begin{remark}
\label{Stembridge-irreducible-circuit-remark}
A different sort of circuit classification in finite root systems was undertaken by
Stembridge \cite{Stembridge}, who defined the notion of an \emph{irreducible circuit}.  Say that a circuit $C=\{\alpha\} \cup I \subset \Phi$ is irreducible if $\alpha$ is in the positive linear span of $I$, and no proper subset of $I$ has any elements of  $\Phi \smallsetminus I$ in its positive linear span.  Stembridge gave a classification, up to isometry, of the irreducible circuits in all finite root systems.  Unfortunately, we did not see how to check Lemma~\ref{acuteness-lemma} directly from the classification of irreducible circuits.  See also Example~\ref{Stembridge-irreducible-circuit-example} below.
\end{remark}

Given a finite reflection group $W$ and a choice of
a root system $\Phi_W$ in $V\cong \RR^n$ for $W$, 
one might attempt to classify all of the 
circuits $C \subset \Phi_W$ up to the action of $W$, that is,
regarding $w(C)$ and $C$ equivalent for all $w$ in $W$.
We will do slightly less, taking advantage of
the following reduction.

\begin{definition}
Call a circuit $C \subset \Phi_W$ a {\it full circuit} if
$\{s_\alpha: \alpha \in C\}$ generates the group $W$.
\end{definition}

Non-full circuits $C \subset \Phi_W$ lie in a root
system $\Phi_{W'}$ for some proper subgroup $W'$ of $W$,
and hence can be dealt with by induction.
Thus we will only classify the full circuits in $\Phi_W$
up to the $W$-action.\footnote{In principle, one could fill in the rest of the classification data using, e.g., the work of Douglass--Pfeiffer--R\"{o}hrle \cite{DouglassPfeifferRohrle}, which classifies the reflection subgroups of finite real reflection groups up to conjugacy.}  Along the way, we will check
that Lemma~\ref{acuteness-lemma} holds in each case.

One further reduction to note is that \emph{only} irreducible root
systems $\Phi_W$ contain full circuits $C$:  if one has 
$V=V_1 \oplus V_2$ with $\Phi=\Phi_1 \sqcup \Phi_2$ and $W = W_1 \times W_2$ 
then the circuit $C \subset \Phi$ being inclusion-minimal
forces $C \subset \Phi_i$ for either $i=1$ or $2$, and hence
$\{s_\alpha: \alpha \in C\} \subset W_i$ for either $i=1$ or $2$.
Thus we only need to consider the irreducible finite root systems.

%%%%%%%%%%%%
%\vskip.1in
%\noindent
%{\sf Rank 1.}
\subsection{Rank 1}
Here $C=\Phi=\{\pm \alpha\}$, whose
acuteness-obtuseness graph has two vertices and a dotted edge:
$$
\xymatrix{+\alpha &\ar[l]\ar[r]& -\alpha} 
\qquad \qquad
\xymatrix@R=.1in@C=.1in{+\alpha \ar@{.}[rrr] & & & -\alpha}
$$

%%%%%%%%%%%%
%\vskip.1in
%\noindent
%{\sf Rank 2: The dihedral groups $I_2(m)$.}
\subsection{Rank 2: the dihedral types $\type{I}_2(m)$}
A full circuit $C = \{\alpha_1, \alpha_2, \alpha_3\}$ 
satisfies $c_1 \alpha_1+c_2 \alpha_2+c_3\alpha_3=0$
for some scalars $c_i$.  Taking the inner product
of both sides of this equation with $c_i \alpha_i$ and noting
that $\langle c_i \alpha_i, c_i \alpha_i \rangle >0$,
one concludes that at most one of the other two inner products 
$\langle c_i \alpha_i, c_j \alpha_j \rangle,
\langle c_i \alpha_i, c_k \alpha_k \rangle$
where $\{i,j,k\}=\{1,2,3\}$
can be positive.  Hence each vertex $i=1,2,3$ 
is incident to at most one edge in the acuteness graph $\Gamma_C$ on
vertex set $\{1,2,3\}$, forcing $\Gamma_C$ to be disconnected---see the typical
pictures below.
\[
\xymatrix@R=.1in@C=.1in{
 & & & & \\
 & & & & \\
 & &\ar[ull]\ar[urr]\ar[dd]& & \\
 & & & & \\
 & & & & 
} \quad
\xymatrix@R=.1in@C=.1in{
c_1 \alpha_1 \ar@{.}[rrrr]& & & & c_2 \alpha_2\ar@{.}[dddll]\\
 & & & & \\
 & & & & \\
 & &c_3 \alpha_3\ar@{.}[uuull]& & 
}
\qquad \qquad
\xymatrix@R=.1in@C=.1in{
  & &  \\
  & &  \\
  &\ar[uul]\ar[uur]\ar[dd]& \\
  & &  \\
  & &  
} \quad
\xymatrix@R=.1in@C=.1in{
c_1 \alpha_1 \ar@{-}[rr]&  & c_2 \alpha_2\ar@{.}[dddl]\\
  & & \\
  & & \\
  &c_3 \alpha_3\ar@{.}[uuul]&  
}
\]

Although the rank $2$ setting required no classification
of the $W$-orbits of full circuits $C \subset \Phi_W$, 
such a classification is not hard.  Consider
the unordered triple $\{A_{12},A_{13},A_{23}\}$, where
$\frac{\pi}{m} A_{ij}$ is the angular measure of the sector 
$\RR_{\geq 0} c_i\alpha_i + \RR_{\geq 0} c_j\alpha_j$,
so that $A_{ij} \in \{1,2,\cdots,m-1\}$ and $A_{12}+A_{13}+A_{23}=2m$.
One checks that $C$ is a full circuit in $\Phi_W$ 
if and only if $g:=\gcd\{A_{12},A_{13},A_{23}\}=1$; otherwise
$C$ is full inside a sub-root system of
type $\type{I}_2(m')$ with $m':=\frac{m}{g}$. 
Furthermore, if $m$ is odd, the unordered triple  $\{A_{12},A_{13},A_{23}\}$
completely determines the $W$-orbit of $C$, while for even $m$, 
there are exactly two $W$-orbits corresponding to each such 
triple, represented by circuits that differ from each other
by a $\frac{\pi}{m}$ rotation.

\begin{remark}
The rank $2$ case raises a reasonable question:  does
the conclusion of  Lemma~\ref{acuteness-lemma} have anything
at all to do with root systems?  In other words, is it possible
that \emph{any} minimal linearly dependent set of
vectors $C=\{\alpha_1,\ldots,\alpha_m\}$ in a Euclidean space $V$
has disconnected acuteness graph $\Gamma_C$?  
Unfortunately, this is \emph{not} true for $\dim(V) \geq 3$.
A result of Fiedler~\cite[Thm.~2.5]{Fiedler}, stated in terms of 
of the \defn{Gram matrix} 
$( \langle \alpha_i, \alpha_j \rangle )_{i,j=1,\ldots,m}$, 
asserts that $C$ will have its \defn{obtuseness graph}
connected, and that one can in fact, find a circuit $C$ with any
prescribed set of obtuse pairs, orthogonal pairs, and acute
pairs, as long as the obtuse pairs form a connected graph.  
When $\dim(V) \geq 3$, this means one can
have both the obtuseness and acuteness graphs being connected.
For example, one has a circuit $\alpha_1+\alpha_2+\alpha_3+\alpha_4=0$
with the following four vectors $(\alpha_i)_{i=1}^4$ in $\RR^3$, 
having acuteness-obtuseness graph as shown:
\[
C=\left(
\alpha_1=\left[ \begin{matrix} -6\\-3\\0\\ \end{matrix} \right], 
\alpha_2=\left[ \begin{matrix} -1\\1\\0\\ \end{matrix} \right],
\alpha_3=\left[ \begin{matrix} 1\\2\\2\\ \end{matrix} \right],
\alpha_4=\left[ \begin{matrix} 6\\0\\-2\\ \end{matrix} \right]
\right)
\qquad \qquad
\raisebox{25pt}{
\xymatrix@R=.2in@C=.2in{
\alpha_1 \ar@{.}[rr]\ar@{-}[dd]&  & \alpha_4\ar@{.}[ddll]\ar@{-}[dd]\\
 & \\
\alpha_2 \ar@{-}[rr]          &  & \alpha_3\ar@{.}[uull]
}}.
\]

\end{remark}

%%%%%%%%%%%%
%\vskip.1in
%\noindent
%{\sf Type $A_{n-1}$ for $n \geq 3$.}
\subsection{Type $\type{A}_{n-1}$ for $n \geq 3$.}
Consider $\RR^n$ with its usual inner product $\langle \cdot, \cdot \rangle$
making the basis vectors $e_1,\ldots,e_n$ orthonormal.
Inside the codimension-one subspace $V=(e_1+\cdots+e_n)^\perp \subset \RR^n$,
considered as a Euclidean space via the restriction of 
$\langle \cdot, \cdot \rangle$, 
one has the \defn{type $\type{A}_{n-1}$ root system}
\[
\Phi_{\type{A}_{n-1}}=\{\pm (e_i - e_j): 1 \leq i < j \leq n\}.
\]
The Weyl group $W$ is the symmetric group $\Symm_n$, permuting
the coordinates in $\RR^n$ and preserving the subspace $V$.
It is well-known and easily checked that full circuits 
in $\Phi_{\type{A}_{n-1}}$ all lie in the $W$-orbit of 
\[
C=\{ 
\alpha_1=e_{1}-e_{2}, \quad 
\alpha_2=e_{2}-e_{3},  \quad
\ldots, \quad
\alpha_{n-1}=e_{n-1}-e_{n}, \quad
\alpha_n=e_{n}-e_{1}
\},
\]
whose minimal dependence is $\alpha_1+\cdots+\alpha_n=0$.  
Since $\langle \alpha_i, \alpha_j \rangle \in \{ -1,  0\}$ for
each $i \neq j$, its acuteness graph $\Gamma_C$ contains $n$ vertices and no edges, and so is disconnected.

Pictorially, one may associate to a subset of $\Phi_{\type{A}_{n-1}}$
a graph on vertex set $\{1,2,\ldots,n\}$ in which the  
roots $\pm (e_i-e_j)$ perpendicular to the hyperplane $x_i=x_j$
are associated with the edge $\xymatrix{i \ar@{-}[r]&j}$.  Circuits
then correspond to graphs that are cycles, and the circuit $C$ above for
$n=4$ would be depicted as the graph on the left, with its acuteness-obtuseness graph shown to its right:
$$
\raisebox{-.05in}{
\xymatrix@R=.3in@C=.3in{
  2\ar@{-}[r]  & 3\ar@{-}[d] \\
  1 \ar@{-}[u] & 4\ar@{-}[l]}
}
\qquad
\xymatrix@R=.1in@C=.1in{
           &+1(e_2-e_3) \ar@{.}[dr]& \\
+1(e_1-e_2) \ar@{.}[ur]&           &+1(e_3-e_4) \ar@{.}[dl]\\
           &+1(e_4-e_1) \ar@{.}[ul]&
}
$$

\begin{remark}
This $W$-orbit of full circuits $C$ in type $\type{A}$ where
$\Gamma_C$ has no edges at all generalizes to an interesting and well-known
family of full circuits for each irreducible crystallographic root system $\Phi$,
which we describe here.  Choose an open fundamental chamber $F$ for
$W$, with corresponding choice of positive roots $\Phi^+$
and simple roots $\Pi$. Then there will always be 
either one or two roots in $\overline{F} \cap \Phi$, namely 
\begin{compactitem}
\item the \defn{highest root} $\alpha_0$,
whose unique expression $\alpha_0=\sum_{i=1}^n c_i \alpha_i$ as
a positive root simultaneously maximizes all the coefficients $c_i$ 
(in particular $c_i > 0$ for each $i=1,\dots,n$), and
\item
the \defn{highest short root} 
$\alpha^* := \left( \alpha_0(\Phi^\vee) \right)^\vee$,
where $\alpha^\vee:=\frac{2\alpha}{|\alpha|^2}$
and $\alpha_0(\Phi^\vee)$ is the highest root for the
{\it dual} crystallographic 
root system $\Phi^\vee:=\{\alpha^\vee: \alpha \in \Phi\}$.
(When $\Phi^\vee=\Phi$, one has $\alpha_0=\alpha^*$.)
\end{compactitem}
Either of the roots $\beta=\alpha_0$ or $\beta=\alpha^*$
gives rise to a full circuit 
$C=\{-\beta\} \sqcup \Pi \subset \Phi$ 
whose minimal dependence has the form
$
-\beta + c_1 \alpha_1 + \cdots c_n \alpha_n = 0.
$
The acuteness graph $\Gamma_C$ has no edges, since
the simple roots are pairwise non-acute and
since $\beta$ in $\overline{F} \cap \Phi$ means that
$\langle \beta ,\alpha_i \rangle \geq 0$ for all $\alpha_i$ in $\Pi$.  
\end{remark}

%%%%%%%%%%%%
%\vskip.1in
%\noindent
%{\sf Types $\type{D}_n$ for $n \geq 3$}.
\subsection{Type $\type{D}_n$ for $n \geq 3$}
The {\it type $\type{D}_n$ root system}
$$
\Phi_{\type{D}_n}=\{\pm e_i \pm e_j: 1 \leq i < j \leq n\}
$$
has Weyl group $W$ which is an index-two subgroup of the
{\it hyperoctahedral group} $\Symm_n^\pm$ of 
all {\it signed permutations} $e_i \mapsto \pm e_{w(i)}$.
Specifically, $W=W(\type{D}_n)$ consists of those signed
permutations in which there are {\it evenly many} indices
$i$ for which $e_i \mapsto -e_{w(i)}$.

Just as one can associate graphs whose edges correspond to pairs $\pm \alpha$ 
of roots in type $\type{A}$, 
Zaslavsky's theory of {\it signed graphs} \cite{Zaslavsky}
associates to each root pair $\pm \alpha$ in $\Phi_{\type{D}_n}$ 
(or reflecting hyperplane $x_i = \pm x_j$)
an edge $\{i,j\}$ on vertex set $\{1,2,\ldots,n\}$
with a $\pm$ label:  
\begin{compactitem}
\item the roots $\alpha=\pm (e_i - e_j)$ with $\alpha^\perp$
defined by $x_i = + x_j$ give rise to \defn{plus edges}
$
\xymatrix{i \ar@{-}[r]^{+}& j}
$,
and
\item the roots $\alpha=\pm (e_i + e_j)$ with $\alpha^\perp$
defined by $x_i = - x_j$ give rise to \defn{minus edges}
$
\xymatrix{i \ar@{-}[r]^{-}& j}.
$
\end{compactitem}
Call a cycle in a signed graph \defn{balanced} if it has an even number of minus edges, and \defn{unbalanced} otherwise.
\begin{proposition}[{Zaslavsky \cite[Thm.~5.1(e)]{Zaslavsky}}]
\label{prop:Zaslavsky}
A set of roots in a root system of classical type is a circuit if and only if its associated signed graph is one of the following types:
\begin{compactenum}[(i)]
\item a balanced cycle,
\item two edge-disjoint unbalanced cycles,
having either a path joining a vertex of one cycle to a vertex of the other,
or else sharing exactly one vertex. 
\end{compactenum}
\end{proposition}

The circuits of type (ii) in Proposition~\ref{prop:Zaslavsky} 
are exemplified by the following full circuits.
Given $i, j \geq 2$ such that $i + j \leq n + 1$, let $C(n;i,j)$ consist of 
two particular unbalanced cycles of sizes $i,j$, connected by a path having $n+1-(i+j)$ edges:
\begin{multline*}
C(n;i,j) 
:= \{ \; e_1 - e_2, \; e_2 - e_3, \; 
           \ldots, \; e_{i - 1} - e_i, \; -e_1 - e_i \; \}
\cup {} \\
\{ \; e_i - e_{i + 1}, \; e_{i + 1} - e_{i + 2}, \; \ldots, \; e_{j - 1} - e_{n-j+1} \; \}
\cup {}\\
\{ \; e_{n-j+1} - e_{n-j + 2}, \; e_{n-j + 2} - e_{n-j + 3}, \; \ldots, \; e_{n - 1} - e_{n}, \; e_{n-j+1} + e_n \; \}.
\end{multline*}
For example, the circuit $C(12;4,6) \subset \Phi_{\type{D}_{12}}$ corresponds to this signed graph: 
$$
\xymatrix@R=.1in@C=.5in{
                &1\ar@{-}[dl]_{+}&                &               &                             &                &8\ar@{-}[dl]_{+}&9\ar@{-}[l]_{+}\\
2\ar@{-}[dr]_{+}&                &4\ar@{-}[ul]_{-}&5\ar@{-}[l]_{+}&6\ar@{-}[l]_{+}&7\ar@{-}[l]_{+}&              &                &10\ar@{-}[ul]_{+}\ar@{-}[dl]^{+}\\
                &3\ar@{-}[ur]_{+}&                &               &                             &                &12\ar@{-}[ul]^{-}&11\ar@{-}[l]^{+}
}
$$
and this acuteness-obtuseness graph:
$$
\tiny
\xymatrix@R=.1in@C=.1in{
 & & & & & &+1(e_8-e_9)\ar@{.}[dr]\ar@{.}[dl]&\\
+1(e_1-e_2)\ar@{.}[r] & +1(-e_1 - e_4)\ar@{-}[dd] \ar@{.}[dr]&          &           &           &+1(e_7-e_8) \ar@{-}[dd] & &+1(e_9-e_{10})\ar@{.}[dd]\\
        &           & +2(e_4 - e_5) \ar@{.}[r] & +2(e_5 - e_6) \ar@{.}[r]& +2(e_6 - e_7) \ar@{.}[ur] \ar@{.}[dr]&       \\
+1(e_2-e_3)\ar@{.}[r] \ar@{.}[uu]& +1(e_3 - e_4) \ar@{.}[ur]&           &           &           &+1(e_7+e_{12}) & &+1(e_{10}-e_{11})\\
& & & & & &+1(e_{11}-e_{12})\ar@{.}[ur]\ar@{.}[ul]&
}
$$
Note that the conditions $i,j \geq 2$ and $i+j \leq n+1$ on $C(n;i,j)$ allow for various degenerate instances, 
including the most degenerate case $C(3;2,2)$ with the following 
signed graph and acuteness-obtuseness graph:
\[
\raisebox{-.17in}{
\xymatrix@R=.5in@C=.5in{
1\ar@/^/@{-}[r]^{+}\ar@/_/@{-}[r]_{-} &2\ar@/^/@{-}[r]^{+}\ar@/_/@{-}[r]_{-}&3
}
}
\qquad \qquad
%\tiny
\xymatrix@R=.2in@C=.3in{
+1(-e_1 - e_2)\ar@{.}[r] \ar@{.}[dr]&+1(e_2-e_3)\\
+1(+e_1 - e_2)\ar@{.}[r] \ar@{.}[ur]&+1(e_2+e_3)
}
\]

The action of the hyperoctahedral group $\Symm^\pm_n$ on subsets of
$\Phi_{\type{D}_n}$ induces an action on their signed graphs
that Zaslavsky calls {\it switching}: the permutations $\Symm_n \subset \Symm_n^\pm$ simply permute the vertex labels on the signed graphs,
while the sign change $e_i \mapsto -e_i$ 
swaps the two kinds of edges incident to vertex $i$, that
is, it swaps $\xymatrix{i \ar@{-}[r]^{+}& j}$ and 
$\xymatrix{i \ar@{-}[r]^{-}& j}$ for any $j$.   Note that this allows one to perform these changes of edge labels in signed graphs via
the switching $e_i \mapsto -e_i$:
\begin{equation}
\label{switches-moving-minusses}
\begin{aligned}
\xymatrix{k \ar@{-}[r]^{-}&i \ar@{-}[r]^{-}& j} 
 &\quad \rightsquigarrow \quad
\xymatrix{k \ar@{-}[r]^{+}&i \ar@{-}[r]^{+}& j} \\
\xymatrix{k \ar@{-}[r]^{+}&i \ar@{-}[r]^{-}& j} 
 &\quad \rightsquigarrow \quad
\xymatrix{k \ar@{-}[r]^{-}&i \ar@{-}[r]^{+}& j} 
\end{aligned}
\end{equation}

\begin{proposition}
\label{type-D-full-circuits-prop}
Consider the set of full circuits in $\Phi_{\type{D}_n}$ under the action
of $\Symm^\pm_n$, and under the action of its subgroup $W(\type{D}_n)$.
A system of orbit representatives for the $\Symm^\pm_n$-action is
\[
\{ C(n;i,j): 2 \leq i \leq j \text{ and } i + j \leq n + 1\}.
\]
Upon restriction to the $W(\type{D}_n)$-action, 
the $\Symm^\pm_n$-orbit of $C(n;i,j)$ 
\begin{compactitem}
\item is a single $W(\type{D}_n)$-orbit if $n$ is odd or if either of $i, j$ is even, and
\item breaks into two $W(\type{D}_n)$-orbits if $n$ is even and both $i, j$ are odd.
\end{compactitem}
\end{proposition}
\begin{proof}
Among the circuits described in Proposition~\ref{prop:Zaslavsky},
the balanced cycles (type (i)) are never full circuits 
in $\Phi_{\type{D}_n}$:  one can use the 
switching action to make them have all plus edges 
$\xymatrix{i \ar@{-}[r]^{+}& j}$, and so the group generated 
by the associated reflections
is conjugate to a subgroup of $\Symm_n \subsetneq W(\type{D}_n)$.

It is easily seen that a circuit of type (ii) in 
Proposition~\ref{prop:Zaslavsky},
having two unbalanced cycles connected by a path, is full in $\Phi_{\type{D}_n}$
if and only if its set of vertices covers $\{1,2,\ldots,n\}$.  
In this case, if its two disjoint cycles have sizes $i, j$ with $i \leq j$, then we claim it is in the 
$\Symm_n^\pm$-orbit of $C(n;i,j)$.  To see this,
perform the following sequence of switchings:
\begin{compactitem}
\item First, apply switches as in \eqref{switches-moving-minusses}
to push all of the minus edges off of the path in the middle, and into the unbalanced cycles at either end.
\item
Then, in each cycle, similarly apply switches to 
push all of the minus edges into one consecutive string, 
touching the unique vertex in the cycle of degree three or more.
\item
Then, in each cycle, apply switches to change pairs of 
consecutive minus edges to plus, so that there is only one minus edge left and it touches the vertex of degree three or more.
\item Finally, apply a permutation in $\Symm_n$ to make the vertex labels match
those of $C(n;i,j)$.
\end{compactitem}

We next analyze the $W$-orbit structure where $W:=W(\type{D}_n)$.  Since $[\Symm^\pm_n:W]=2$, any $\Symm^\pm_n$-orbit is either a single $W$-orbit,
or splits as a union of two $W$-orbits.  
One way to show that a $\Symm_n^\pm$-orbit remains a single $W$-orbit is to exhibit
an element $C$ of the orbit and some $w$ in $\Symm^\pm_n \smallsetminus W$ with
$w(C)=C$.  Note that any circuit $C$ is fixed by the element $w_0$ in $\Symm^\pm_n$
that sends $e_i \mapsto -e_i$ for all $i=1,2,\ldots,n$, and when $n$ is odd, 
$w$ lies in $\Symm^\pm_n \smallsetminus W$.  Thus no $\Symm_n^\pm$-orbits split when $n$ is odd.
Also, if $i$ is even, then the circuit $C(n;i,j)$ is fixed by the element $w$ in $\Symm^\pm_n \smallsetminus W$ that sends $e_k \leftrightarrow - e_{i - k}$ for $1 \leq k \leq i - 1$ (in particular $e_{\frac{i}{2}} \mapsto -e_{\frac{i}{2}}$).
Thus the $\Symm^\pm_n$-orbit of $C(n; i,j)$ does not split when $i$ is even.  A similar argument shows that it does not split when $j$ is even.

It only remains to show that the $\Symm^\pm_n$-orbit of $C(n;i,j)$
\emph{does} split into two $W$-orbits when $n$ is even but $i, j$ are both odd.
To do this, we describe a $\ZZ/2\ZZ$-valued $W$-invariant $\pi(C)$ of these circuits $C$.
Consider the unique perfect matching $M$ of the undirected graph for $C$.  For example, $M$ is shown here as the doubled edges for $(n, i, j) = (16, 5, 7)$:
$$
\xymatrix@R=.2in@C=.2in{
\bullet \ar@{=}[r]\ar@{-}[dd] & \bullet  \ar@{-}[dr] & & & & & & &\bullet \ar@{=}[r]&\bullet \ar@{-}[r]&\bullet \ar@{=}[dd]\\
 &  &\bullet \ar@{=}[r] &\bullet \ar@{-}[r] &\bullet \ar@{=}[r] &\bullet \ar@{-}[r] &\bullet\ar@{=}[r] &\bullet \ar@{-}[ur] \ar@{-}[dr]\\
\bullet \ar@{=}[r] & \bullet  \ar@{-}[ur] & & & & & & &\bullet \ar@{=}[r]&\bullet \ar@{-}[r]&\bullet
}
$$
Define $\pi(C)$ to be the parity of the number of minus 
edges in the signed graph for $C$ that lie in $M$.
Applying elements of $\Symm_n$ to $C$ does not affect $\pi(C)$,
but switches of the form $e_k \mapsto -e_k$ reverse $\pi(C)$.  Thus both values 
$\pi(C)$ in $\ZZ/2\ZZ$ occur within the $\Symm^\pm_n$-orbit of $C(n;i,j)$, while only one value occurs in each $W$-orbit.
\end{proof}

Note that Proposition~\ref{type-D-full-circuits-prop} immediately implies that full circuits $C \subset \Phi_{\type{D}_n}$ have disconnected acuteness graph $\Gamma_C$, since $\Gamma_{C(n;i,j)}$ has at least four vertices but at most two edges.

%%%%%%%%%%%%
%\vskip.1in
%\noindent
%{\sf Type $B_n/\type{C}_n$ for $n \geq 2$.}
\subsection{Type $\type{B}_n/\type{C}_n$ for $n \geq 3$} 
Since we are only concerned with the hyperplanes and reflections 
associated to the roots, we are free to choose the crystallographic
root system of type $\type{C}_n$:
\[
\Phi_{\type{C}_n}:=\{\pm e_i \pm e_j: 1\leq i < j \leq n\} \sqcup \{ \pm 2e_i: 1 \leq i \leq n\}.
\]
Here $W=W(\type{C}_n)=\Symm_n^\pm$ is the full hyperoctahedral group of $n \times n$ signed permutations $e_i \mapsto \pm e_{w(i)}$.

As in type $\type{D}$, Zaslavsky \cite{Zaslavsky} associates a signed graph to each subset of roots.  Roots in $\Phi_{\type{D}_n}$ correspond to (signed) edges as before, and the pair $\pm 2e_i$ is depicted as a {\it self-loop} on vertex $i$, with a minus sign.  Such a self-loop is considered an unbalanced cycle (with one edge).  Then Proposition~\ref{prop:Zaslavsky} remains correct as a characterization of circuits $C \subset \Phi_{\type{C}_n}$, that is, they are either of type (i) or (ii) mentioned there, allowing for self-loops as unbalanced cycles.

We thus extend the definition of the circuits $C(n;i,j)$ to allow $C(n;1,j)$ where $1 \leq j \leq n$:
\begin{multline*}
C(n; 1, j) := 
\{\; -2e_1 \;\} \; \cup \;
\{ \; e_1 - e_2, \; e_2 - e_3, \; \ldots, \; e_{n -j} - e_{n-j+1} \; \}
\cup {} \\
\{ \; e_{n-j+1} - e_{n-j+2}, \; e_{n-j+2} - e_{n-j+3}, \; \ldots, \; e_{n-1} - e_n \; e_{n-j+1}+e_n\}.
\end{multline*}
The following example depicts $C(9;1,6)$ as a signed graph, as well as its acuteness-obtuseness graph:
\[
\qquad
\xymatrix@R=.1in@C=.5in{
                &               &                             &                &5\ar@{-}[dl]_{+}&6\ar@{-}[l]_{+}\\
1 \ar@(ul,dl)@{-}_{-}&2\ar@{-}[l]_{+}&3\ar@{-}[l]_{+}&4\ar@{-}[l]_{+}&              &                &7\ar@{-}[ul]_{+}\ar@{-}[dl]^{+}\\
                &               &                             &                &9\ar@{-}[ul]^{-}&8\ar@{-}[l]^{+}
}
\]
\[
\tiny
\xymatrix@R=.1in@C=.1in{
& & & & &+1(e_5-e_6)\ar@{.}[dr]\ar@{.}[dl]&\\
&         &           &           &+1(e_4-e_5) \ar@{-}[dd] & &+1(e_6-e_7)\ar@{.}[dd]\\
+1(-2e_1) \ar@{.}[r] & +2(e_1 - e_2) \ar@{.}[r]& +2(e_2 - e_3) \ar@{.}[r]& +2(e_3 - e_4) \ar@{.}[ur] \ar@{.}[dr]&       \\
&          &           &           &+1(e_4+e_{9}) & &+1(e_7-e_8)\\
& & & & &+1(e_9-e_8)\ar@{.}[ur]\ar@{.}[ul]&
}
\]
Note that the condition $1 \leq j \leq n$ on $C(n;1,j)$ allows for various degenerate instances. 
As examples, in the case $j=1$ we have the circuit $C(4;1,1)$
$$
\quad
\xymatrix@R=.5in@C=.5in{
1 \ar@(ul,dl)@{-}_{-}&2\ar@{-}[l]_{+}&3\ar@{-}[l]_{+}&4\ar@{-}[l]_{+} \ar@(ur,dr)@{-}^{-}
}
$$
with acuteness-obtuseness graph
$$
\small
\xymatrix@R=.2in@C=.2in{
+1(-2e_1)\ar@{.}[r] &
  +2(e_1 - e_2) \ar@{.}[r]&
  +2(e_2 - e_3) \ar@{.}[r]&
  +2(e_3 - e_4) \ar@{.}[r] &
  +1(+2e_1)
}
$$
and in the case $j=n$ we have the circuit $C(5;1,5)$:
$$
\xymatrix@R=.2in@C=.3in{
 & 2 \ar@{-}[r]^{+} & 3 \ar@{-}[dd]^{+} \\
1 \ar@(ul,dl)@{-}_{-} \ar@{-}[ur]^{+}  \ar@{-}[dr]_{-}&  &\\
 & 5 \ar@{-}[r]_{+} & 4
}
\qquad \qquad
\small
\xymatrix@R=.2in@C=.3in{
 & +1(e_1 - e_2) \ar@{.}[r]  \ar@{-}[dd]&+1(e_2 - e_3)\ar@{.}[dr]& \\
+1(-2e_1)\ar@{.}[ur] \ar@{.}[dr]& &             &+1(e_3 - e_4) \\
& +1(e_1 + e_5) \ar@{.}[r]&+1(e_4 - e_5)\ar@{.}[ur]& \\
}
$$

\begin{proposition}
\label{type-C-full-circuits-prop}
The set
$
\{ C(n;1,j): 1 \leq j \leq n\}
$
is a system of representatives for the $\Symm^\pm_n$-orbits of full circuits in $\Phi_{\type{C}_n}$.
\end{proposition}
\begin{proof}
As before, among the circuits described in Proposition~\ref{prop:Zaslavsky},
those of type (i) (balanced cycles) are never full circuits.  But now a circuit of type (ii),
having two unbalanced cycles connected by a path, is full
if and only if its set of vertices covers $\{1,2,\ldots,n\}$ \emph{and also}
 one of its balanced cycles has size one, i.e., is a self-loop.  
In this case, if its two disjoint cycles have sizes $1, j$, then we claim it is in the 
$\Symm_n^\pm$-orbit of $C(n;1,j)$.  To see this,
perform switchings as in type $\type{D}$ that push all of the minus edges off of the path in the middle and into the unbalanced cycle of size $j$, then  toward one end of this cycle, and cancel them in pairs until only one is left;  the result can then be relabeled by an element of $\Symm_n$ to give $C(n;1,j)$.
\end{proof}

Note that Proposition~\ref{type-C-full-circuits-prop} immediately implies that full circuits $C \subset \Phi_{\type{C}_n}$ have disconnected acuteness graph $\Gamma_C$, since $\Gamma_{C(n;1,j)}$ has at least three vertices but at most one edge.

%%%%%%%%%%%%
%\vskip.1in
%\noindent
%{\sf Exceptional types $\type{H}_3,\type{H}_4,\type{F}_4, \type{E}_6, \type{E}_7,\type{E}_8$.}
\subsection{Exceptional types}
\label{exceptional-type-circuit-section}
We outline our {\tt Mathematica} computations verifying  Lemma~\ref{acuteness-lemma} in the exceptional types $\type{H}_3$, $\type{H}_4$, $\type{F}_4$, $\type{E}_6$, $\type{E}_7$, and $\type{E}_8$.
This data is attached as auxilliary data files named in a logical way; e.g., data for type $\type{E}_8$ is in the file \texttt{E8.txt}. 
We first generated a set of $W$-orbit representatives for all bases of positive roots in each root system $\Phi_W$. 
Given the list of $W$-orbit representatives for bases $B \subset \Phi_W^+$, we produced the $W$-orbit representatives for all circuits $C$ by adding to each $B$ a positive root $\alpha \in \Phi_W^+ \smallsetminus B$ in all
possible ways, finding the unique circuit $C \subset B \cup \{\alpha\}$, and classifying all such $C$ up to $W$-action.  Non-full circuits were discarded.
Finally, for each of these full circuits
$C$, we computed the acuteness graph $\Gamma_C$ and verified that it was disconnected.

The table below shows the number of orbits of bases and of full circuits in each of the exceptional types.
\begin{center}
\begin{tabular}{ c | c | c }
Type & \# orbits of bases & \# orbits of full circuits \\
\hline
$\type{H}_3$ &   11 &   15 \\ \hline
$\type{H}_4$ &   96 &  416 \\ \hline
$\type{F}_4$ &   35 &   22 \\ \hline
$\type{E}_6$ &   39 &   17 \\ \hline
$\type{E}_7$ &  311 &  142 \\ \hline
$\type{E}_8$ & 1943 & 1717
\end{tabular}
\end{center}
In the case of $\type{E}_8$, this computation required several days to produce the $1943$ $W$-orbits of bases in $\Phi_{\type{E}_8}^+$.  To corroborate this data, we also produced the sizes of the stabilizers of each $W$-orbit representative $B$; these allowed us to compare with the calculations of De Concini--Procesi \cite{DeConciniProcesi}, who found (e.g.)~that there are $348607121625$ total bases in $\Phi_{\type{E}_8}^+$.

\begin{example}
\label{Stembridge-irreducible-circuit-example}
Some of the full circuits that we encountered are
the irreducible circuits $C=\{\alpha\} \cup I$,
discussed in Remark~\ref{Stembridge-irreducible-circuit-remark} above.  
Stembridge shows %
%See Thm. 2.2(a) and Prop. 3.1(a).  
% (Stembridge asserts in \S4 that the latter is valid in noncrystallographic types as well.
that what he calls the {\it apex} vector $\alpha$
has $\langle \alpha, \beta \rangle > 0$ for each $\beta$ in $I$,
if the irreducible circuit $C$ comes from a dependence 
of the form $(-1)\alpha+\sum_{\beta \in I} c_\beta \beta = 0$ with
$c_\beta > 0$.  Therefore $\alpha$ always gives rise to a vertex of $\Gamma_C$
having an obtuse edge to every other vertex in the acuteness-obtuseness graph, 
and so becomes an isolated vertex of the acuteness graph $\Gamma_C$.
We depict here the acuteness-obtuseness graphs for the
irreducible circuits in type $\type{E}_6,\type{E}_7,\type{E}_8$,
adapted from his figure \cite[Fig.~1]{Stembridge},
where the apex $\alpha$ always appears in the center:
$$
\type{E}_6
\tiny
\xymatrix@R=.1in@C=.1in{
 &\alpha_2& &\alpha_3\ar@{-}[dr]&\\
\alpha_1\ar@{-}[ur]& &(-3)\alpha
\ar@{.}[rr]\ar@{.}[ur]\ar@{.}[ll]\ar@{.}[ul]\ar@{.}[dr]\ar@{.}[dl]
& &\alpha_4 \\
&\alpha_6& &\alpha_5\ar@{-}[ll]&
} \qquad \qquad
\type{E}_7
\tiny
\xymatrix@R=.1in@C=.1in{
 \alpha_1\ar@{-}[dd]\ar@{-}[dr]&              &+2\alpha_7&             &\alpha_4\ar@{-}[dd]\\
 &\alpha_2&(-4)\alpha
\ar@{.}[u] \ar@{.}[l] \ar@{.}[r] \ar@{.}[ull] \ar@{.}[urr] \ar@{.}[dll] \ar@{.}[drr]
    &\alpha_5 \ar@{-}[dr] \ar@{-}[ur]& \\
 \alpha_3\ar@{-}[ur]&              &                  &             &\alpha_6\ar@{-}[ul]
}
$$
\vskip.2in
$$
\type{E}_8
\tiny
\xymatrix@R=.1in@C=.05in{
 &+2\alpha_1 \ar@{-}[dl] & &\alpha_5 \ar@{-}[dddd] \ar@{-}[dr] \ar@{-}[dddr]& \\
\alpha_2  \ar@{-}[dd]& & & &\alpha_6 \ar@{-}[dd] \ar@{-}[dddl]\\
& &(-5)\alpha
\ar@{.}[uur] \ar@{.}[urr] \ar@{.}[drr] \ar@{.}[ddr] 
\ar@{.}[uul] \ar@{.}[ull] \ar@{.}[ddl]\ar@{.}[dll]
& & \\
\alpha_3& & & &\alpha_7\\
 &+2\alpha_4\ar@{-}[ul] & &\alpha_8\ar@{-}[ur]&
}
\qquad 
\type{E}_8
\tiny
\xymatrix@R=.1in@C=.05in{
& &              & \alpha_4\ar@{-}[dddd]& & &\\
\alpha_1\ar@{-}[dd]& &              & & &\alpha_5\ar@{-}[ull]& \\
&\alpha_2\ar@{-}[ul]&(-4)\alpha
\ar@{.}[l] \ar@{.}[ull] \ar@{.}[dll] 
\ar@{.}[uur] \ar@{.}[urrr] \ar@{.}[rrrr] \ar@{.}[drrr]\ar@{.}[ddr]
& & & &\alpha_6\ar@{-}[ul]\\
\alpha_3\ar@{-}[ur]& &              & & & \alpha_7\ar@{-}[ur]& \\
& &              &  \alpha_8\ar@{-}[urr]& & &
}
\qquad 
\type{E}_8
\tiny
\xymatrix@R=.1in@C=.05in{
           &+3 \alpha_1&              & \alpha_4
                                        \ar@{-}[dddd] 
                                        \ar@{-}[dddrrr] 
                                        \ar@{-}[ddrrrr] 
                                        \ar@{-}[drrr]   &&&            &        \\
           &           &              &                 &&&\alpha_5
                                                           \ar@{-}[dd]
                                                           \ar@{-}[dr] &        \\
           &           &(-6)\alpha
                        \ar@{.}[uul] 
                        \ar@{.}[ddl] 
                        \ar@{.}[dll] 
                        \ar@{.}[uur] 
                        \ar@{.}[urrr] 
                        \ar@{.}[rrrr] 
                        \ar@{.}[drrr]
                        \ar@{.}[ddr]  &                 &&&            &\alpha_6\\
+2\alpha_2
\ar@{-}[dr]&            &             &                 &&& \alpha_7
                                                           \ar@{-}[ur]&        \\
           &+2\alpha_3  &              & \alpha_8 
                                         \ar@{-}[uuurrr]
                                         \ar@{-}[urrr] 
                                         \ar@{-}[uurrrr]&&&            &
}
\begin{comment}
\xymatrix@R=.1in@C=.05in{
 &+3 \alpha_1&              &  \alpha_4\ar@{-}[dddd] \ar@{-}[dddrr] \ar@{-}[ddrrr] \ar@{-}[drr]&&\\
 & &              &   & &\alpha_5\ar@{-}[dd]\ar@{-}[dr]&\\
 & &(-6)\alpha
\ar@{.}[uul] \ar@{.}[ddl] \ar@{.}[dll] 
\ar@{.}[uur] \ar@{.}[urrr] \ar@{.}[rrrr] \ar@{.}[drrr]\ar@{.}[ddr]
&   &  &            &\alpha_6\\
+2\alpha_2\ar@{-}[dr]& &              &  & & \alpha_7\ar@{-}[ur]&\\
 &+2\alpha_3&              &  \alpha_8 \ar@{-}[uuurr]\ar@{-}[urr] \ar@{-}[uurrr]&&
}
\end{comment}
$$
\end{example}

\begin{example}
\label{favorite-full-circuits-remark}
 In most cases it is extremely easy to recognize the 
disconnectedness of $\Gamma_C$:  either it has an isolated vertex, or  it has $v$ vertices and fewer than $v-1$ edges, or both.  For example, 
in Stembridge's irreducible circuits $C=\{\alpha\} \cup I$, the apex vector
$\alpha$ necessarily gives rise to an isolated vertex of $\Gamma_C$.

Meanwhile in type $\type{H}_4$, of the $419$ full-rank circuit orbit 
representatives, there are only $25$ with at least four edges (two each with five or six edges, $21$ with four edges); of these, ten have an isolated vertex (including all of those with more than four edges) and the other fifteen consist of a disjoint triangle and edge.
\end{example}

\begin{example}
Here is another interesting example of an
acuteness-obtuseness graph of a full circuit in $\type{E}_8$:
\begin{comment}
\[
\xymatrix@R=.1in@C=.05in{
 &+2\alpha_* \ar@{-}[dl]\ar@{.}[rrr]\ar@{.}[drrrr]\ar@{.}[dddrrrr]\ar@{.}[ddddrrr] &&&+1\alpha_* \ar@{-}[dddd] \ar@{-}[dr] \ar@{-}[dddr]& \\
+1\alpha_*  \ar@{-}[dd]\ar@{.}[urrrr]\ar@{.}[rrrrr]\ar@{.}[ddrrrrr]\ar@{.}[dddrrrr]& && & &+1\alpha_* \ar@{-}[dd] \ar@{-}[dddl]\\
& &&+1\alpha_*
\ar@{-}[uur] \ar@{-}[urr] \ar@{-}[drr] \ar@{-}[ddr] 
\ar@{.}[uull] \ar@{.}[ulll] \ar@{.}[ddll]\ar@{.}[dlll]
& & \\
+1\alpha_*\ar@{.}[uuurrrr]\ar@{.}[uurrrrr]\ar@{.}[rrrrr]\ar@{.}[drrrr]& && & &+1\alpha_*\\
 &+2\alpha_*\ar@{-}[ul]\ar@{.}[rrr]\ar@{.}[urrrr]\ar@{.}[uuurrrr]\ar@{.}[uuuurrr] && &+1\alpha_*\ar@{-}[ur]&
}
\]
\end{comment}
\[
\xymatrix@R=.1in@C=.5in{
+2\alpha_1 \ar@{-}[dd]
  \ar@{.}[ddrrr]\ar@{.}[drrrr]\ar@{.}[dddrrrrr]\ar@{.}[ddddrrr]\ar@{.}[dddddrrrr]    
& & & & \\
& & & &\alpha_6\ar@{-}[ddr] \ar@{-}[dddl] \ar@{-}[dddd] &
\\
\alpha_2 \ar@{-}[dd] 
  \ar@{.}[rrr]\ar@{.}[urrrr]\ar@{.}[drrrrr]\ar@{.}[ddrrr]\ar@{.}[dddrrrr]   
  & & &\alpha_5\ar@{-}[ur]\ar@{-}[dd] \ar@{-}[dddr]\ar@{-}[drr]& & \\
           & & & & &\alpha_7\ar@{-}[dll] \\
\alpha_3 
  \ar@{.}[rrr]\ar@{.}[drrrr]\ar@{.}[urrrrr]\ar@{.}[uurrr]\ar@{.}[uuurrrr]   
    \ar@{-}[dd] & & &\alpha_9\ar@{-}[dr] &  & \\
  & & & &\alpha_8\ar@{-}[uur] & \\
+2\alpha_4 
  \ar@{.}[uurrr]\ar@{.}[urrrr]\ar@{.}[uuurrrrr]\ar@{.}[uuuurrr]\ar@{.}[uuuuurrrr]    
& & & &
}
\]
\end{example}

\begin{comment}
\footnote{
\begin{center}
\begin{tabular}{ c | c | p{2in} }
% \textrm{Type} & \textrm{Number of orbits of bases} & \textrm{Stabilizer sizes} \\
Type & Number of orbits of bases & Stabilizer sizes \\
\hline
$\type{H}_3$ &   11 & $2^4, 4^4, 12^2, 24^1$ \\ \hline
$\type{H}_4$ &   96 & $2^{31},4^{39},6^{1},8^{14},12^{2},16^{3},24^{2},48^{3},192^{1}$ \\ \hline
$\type{F}_4$ &   35 & $2^6, 4^7, 8^9, 12^4, 16^4, 24^2, 64^1, 384^2$ \\ \hline
$\type{E}_6$ &   39 &  $1^{6}, 2^{15}, 4^{8}, 6^{1}, 8^{2}, 10^{1}, 12^{2}, 16^{2}, 48^{1}, 240^{1}$ \\ \hline
$\type{E}_7$ &  311 &  
    $2^{77}, 4^{112}, 8^{55}, 12^{16}, 16^{10}, 20^{1}, 24^{13}, 28^{1}, $
    $32^{4}, 40^{1}, 48^{8}, 64^{1}, 72^{1}, 96^{4}, 128^{1}, 144^{1}, $
    $240^{1}, 288^{1}, 480^{1}, 10080^{1}, 21504^{1} $\\ \hline
$\type{E}_8$ & 1943 & 
    $2^{544}, 4^{679}, 8^{330}, 12^{65}, 16^{97}, 20^{2}, 24^{83}, $
    $28^{1}, 32^{29}, 40^{4}, 48^{37}, 56^{1}, 64^{5}, 72^{1}, 80^{1}, $
    $96^{22}, 120^{2}, 128^{3}, 144^{3}, 192^{4}, 240^{3}, 256^{2}, $
    $288^{5}, 384^{1}, 480^{7}, 512^{1}, 768^{2}, 960^{1}, 1024^{1}, $
    $1152^{1}, 1440^{2}, 2304^{1}, 20160^{1}, 80640^{1}, 344064^{1}$
\end{tabular}
\end{center}
Total number of bases:
H3: 385 (directly computed)
F4: 7560 (D--P, directly computed)
H4: 398475 (directly computed)
E6: 846720 (D--P)
E7: 221714415 (D--P)
E8: 348607121625 (D--P)

Size of Weyl group:
H3: 120
F4: 1152
H4: 14400
E6: 51840
E7: 2903040
E8: 696729600

}
)
\end{comment}

%%%%%%%%%%%%%%%%%%%%%%%%%%%%%%%%%%%%%%%%%%%%%%%%%%%%%%%%%%%%%
\section{Non-minimal factorizations, and
proofs of Lemma~\ref{main lemma} and Corollary~\ref{cor:standard form} 
}
\label{main lemma section}
%%%%%%%%%%%%%%%%%%%%%%%%%%%%%%%%%%%%%%%%%%%%%%%%%%%%%%%%%%%%%

Most of this section is devoted to the proof of the following lemma from the Introduction,
recalled here, which we then use to prove Corollary~\ref{cor:standard form}.

\vskip.1in
\noindent
{\bf Lemma~\ref{main lemma}.}
{\it
For any reflection factorization $\boldt = (t_1, \ldots, t_m)$ of
$w = t_1 \cdots t_m$ with $\ell_T(w) < m$,
either $m = 2$ or there exists $\boldt' = (t'_1, \ldots, t'_m)$ 
in the Hurwitz orbit of $\boldt$ with 
$\ell_T(t'_1 \cdots t'_k) < k$ for some $k \leq m-1$.
}
\vskip.1in
\noindent

An important tool will be Carter's characterization of
minimal reflection factorizations.

\begin{proposition}[{Carter \cite[Lem.~3]{Carter}}]
\label{Carter's lemma}
In a finite real reflection group $W$, one has
$\ell_T(s_{\alpha_1} \cdots s_{\alpha_k})=k$ if and only the roots 
$\alpha_1, \ldots, \alpha_k$ are linearly independent.
\end{proposition}

\noindent
In particular, this implies that the reflection length function 
$\ell_T: W \rightarrow \NN$ only takes values in $\{0,1,\ldots,\dim(V)\}$.
A second important observation is the following.

\begin{proposition}
\label{prefix-prop}
A subsequence $(t_{i_1},\ldots,t_{i_k})$ with 
$1 \leq i_1 < \ldots < i_k \leq m$ of 
a factorization $\boldt=(t_1,\ldots,t_m)$ of $w=t_1 t_2 \cdots t_m$ 
is always a prefix for some 
$
\boldt'=(t_{i_1},\ldots,t_{i_k},t'_{k+1},t'_{k+2},\ldots,t'_m)
$ 
in the Hurwitz orbit of $\boldt$.
\end{proposition}
\begin{proof}
Starting with $\boldt$, 
apply $\sigma_{i_1-1}, \sigma_{i_1-2}, \ldots, \sigma_2, \sigma_1$
to move the $t_{i_1}$ to the first position; then similarly 
apply $\sigma_{i_2-1}, \sigma_{i_2-2}, \ldots, \sigma_3, \sigma_2$ to move
$t_{i_2}$ to the second position, and so on.
\end{proof}

Using Propositions~\ref{Carter's lemma} and \ref{prefix-prop}
in the context of Lemma~\ref{main lemma}, one
can assume without loss of generality that the
reflection factorization $\boldt = (t_1, \ldots, t_m)$ of
$w = t_1 \cdots t_m$ with $\ell_T(w) <m$
corresponds via $t_i = s_{\alpha_i}$ to a sequence of roots 
$\alpha_1,\alpha_2,\ldots,\alpha_m$ that form a circuit
$C=\{\alpha_1,\ldots,\alpha_m\} \subset \Phi$.
Furthermore, as in Section~\ref{section: acuteness},
one can also assume that $\Phi_W$ is irreducible,
and that $C$ is a full circuit in $\Phi_W$.

Note that Lemma~\ref{main lemma} can be checked trivially in rank $1$, 
since $W=\langle s: s^2=e \rangle$.  The next subsection
deals with rank $2$, and the one following deals with
ranks $3$ and higher, relying ultimately on Lemma~\ref{acuteness-lemma}.

%%%%%%%%%%%%%%%%%%
\subsection{Rank 2}
\label{rank-two-standard-form-section}

Lemma~\ref{main lemma} is already interesting in rank $2$,
so that $W$ is the dihedral group
$$
W=W_m:=\langle s,t: s^2=t^2=e, (st)^m=e \rangle
$$
of type $\type{I}_2(m)$.  Since full circuits $C$ have size $3$,
the reductions above show that, to prove Lemma~\ref{main lemma},
it only remains to check the following assertion:
any reflection
factorization $\boldt =(t_1, t_2, t_3)$ of $w=t_1 t_2 t_3$ in $W_m$
has a factorization of the form $\boldt'=(t', t', t'')$
in its Hurwitz orbit.
In fact, one need only prove this same assertion for the
\defn{infinite dihedral group} 
\[
W_\infty=\langle s, t \mid s^2 = t^2 = e\rangle
\]  
of type $\type{I}_2(\infty)$.  The reflections in $W_\infty$ are the elements of odd length in the generating set $s, t$; we denote them by $
T=\{
t(n):= (st)^n s, n \in \ZZ
\}
$.  (In particular, this means $s = t(0)$ and $t = t(-1)$.)  The obvious surjection $W_\infty \twoheadrightarrow W_m$ 
\begin{compactitem}
\item
carries reflections in $W_\infty$ to reflections in $W_m$, and
\item
sends Hurwitz moves $\sigma_i$ on factorizations in $W_\infty$
as in \eqref{Hurwitz-move-definition} to the same  Hurwitz move in $W_m$.
\end{compactitem}

There is a standard geometric model for $W_\infty$ as generated by
{\it affine reflections} of the real line $\RR$: the reflection $t(n)$ reflects $\RR$ across the point $x=n$ in $\RR$.
Thus the conjugation action 
\[
t(b)^{t(a)} = t(a) \cdot t(b) \cdot t(a) = t(2a - b)
\]
corresponds to reflecting $b$ across $a$ on the line $\RR$.
Therefore, bearing in mind Proposition~\ref{prefix-prop},
it will suffice to show that, given an ordered 
triple of integers $(a,b,c)$,
one can eventually reach a triple having two of the integers equal
via moves that reflect one of $a,b$ across the other and swapping
their positions within the triple,
or doing the same with $b,c$.
We give an algorithm that does this by reflecting one of the three
values $a,b,c$ across the median value, proceeding by induction on the
(positive integer) length 
$$
M(a,b,c):=\max\{a,b,c\}-\min\{a,b,c\}
$$
of the interval that they span, and eventually making two of them coincide.

Up to the irrelevant symmetries $(a,b,c) \mapsto (c,b,a)$ and
$(a,b,c) \mapsto (-a,-b,-c)$ 
(the latter being achieved by reflection across $0$),
we may suppose that either $a \leq b \leq c$ or $a \leq c < b$.  
In the latter case, reflecting $b$ across $a$ produces  
$(2a-b,a,c)$ with $2a - b < a \leq c$, so we reduce to the former case.
Since $a \leq b \leq c$, one has $M=c-a$.  Let $m:=\min\{b-a, c-b\}$.
Without loss of generality, $M(a,b,c) > m > 0$, else we are done.
If $m=b-a$, reflect $a$ across $b$ giving
$(a',b',c')=(b,2a-b,c)$ with 
$
M(a',b',c')=c-b\leq m<M(a,b,c).
$
If $m=c-b$, reflect $c$ across $b$ giving $(a',b',c')=(a,2b-c,b)$ with 
$
M(a',b',c')=b-a \leq m<M(a,b,c).
$
In either case, we are done by induction. 

Here is an illustration in the case $(a, b, c) = (3, 7, 5)$ 
(so that initially $a<c<b$):
\begin{figure}[h]
%\[
%\begin{array}{ccccccc}
%(t(3), t(7), t(5)) 
%  & \overset{\sigma_1^{-1}}{\longrightarrow}
%  & (t(-1), t(3), t(5)) 
%  & \overset{\sigma_2^{-1}}{\longrightarrow}
%  & (t(-1), t(1), t(3)) 
%  & \overset{\sigma_1}{\longrightarrow}
%  & (t(1), t(3), t(3))
%\end{array}
%\]
%\begin{center}
  \include{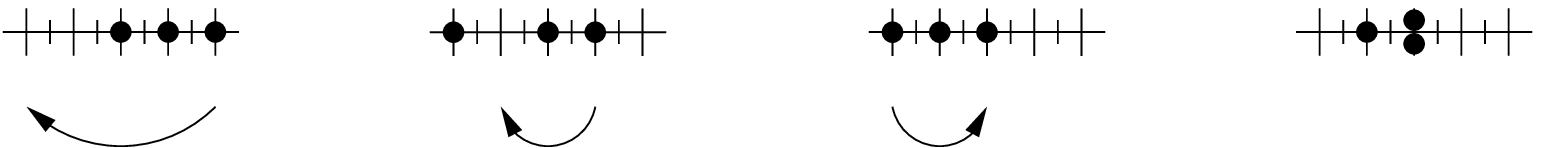} 
%\end{center}
%\caption{The application of several Hurwitz moves to a triple of three reflections in $\type{I}_2(\infty)$ produces a triple in which two elements are equal.}
%\label{dihedral Hurwitz example}
\end{figure}

\subsection{Higher ranks}

When $W$ has rank at least three, 
we require a somewhat more subtle argument to prove Lemma~\ref{main lemma}.
Given a factorization $w=t_1 t_2 \cdots t_m$ in which $\ell_T(w) < m$,
there exists an $m$-tuple $(\alpha_1, \ldots,\alpha_m)$ of roots 
for which $t_i =s_{\alpha_i}$, and by Proposition~\ref{Carter's lemma} 
this $m$-tuple is linearly dependent.

\begin{definition}
A pair $(\CCC,\ccc)$ where
$\CCC=(\alpha_1,\ldots,\alpha_m)$ in $\Phi^m$ 
and $\ccc=(c_1,\ldots,c_m)$ in $\RR^m$ with  
$\sum_{i=1}^m c_i \alpha_i=0$ will be called an \defn{$m$-dependence} in $\Phi$.
%We will sometimes picture $(\CCC,\ccc)$ as a biword
%$
%\begin{pmatrix}
%\alpha_1 & \cdots & \alpha_m \\
%c_1      & \cdots & c_m
%\end{pmatrix}
%$
Its \defn{weight} is defined as
$$
\wt(\CCC,\ccc):=\wt(\ccc):=\sum_{i=1}^m |c_i|.
$$
\end{definition}

Our proof strategy for Lemma~\ref{main lemma} is to start with any
nontrivial $m$-dependence $(\CCC,\ccc)$ that accompanies
a non-minimal factorization $w=t_1 t_2 \cdots t_m$, 
and try to apply Hurwitz moves that make $\wt(\CCC,\ccc)$ strictly
smaller.  Then we work by induction to show that for $m > 2$, every $m$-dependence has in its Hurwitz orbit an $m$-dependence
$(\CCC',\ccc')$ where one of the coefficients $c'_i$ {\it vanishes}, so
that a proper subset of the vectors in $\CCC$ is dependent.  
Bearing in mind Proposition~\ref{prefix-prop}, this would prove Lemma~\ref{main lemma}.  There are at least three separate issues here:
\begin{compactenum}[(i)]
\item We need a well-defined Hurwitz action on the set of $m$-dependences (easy---see Proposition~\ref{Hurwitz-action-on-dependences}).
\item We need to know that some Hurwitz move applies that {\it lowers} 
$\wt(\CCC,\ccc)$.  Here we use Lemma~\ref{acuteness-lemma}.
\item We need to know that one cannot lower
$\wt(\CCC,\ccc)$ infinitely often. This is a fairly easy argument
in the crystallographic case, but requires one further computation 
in types $\type{H}_3, \type{H}_4$.
\end{compactenum}
We deal with issues (i), (ii), (iii) in the next three subsections.

%%%
\subsubsection{Dealing with issue (i).}

We lift Hurwitz moves on reflection factorizations to moves on $m$-dependences.

\begin{proposition}
\label{Hurwitz-action-on-dependences}
The Hurwitz move 
$\boldt \overset{\sigma_i}{\longmapsto} \boldt'$
of \eqref{Hurwitz-move-definition}
lifts to the following (invertible) Hurwitz move $\sigma_i$ on the set of
$m$-dependences in $\Phi$:  given
$(\CCC=(\alpha_i)_{i=1}^m,\ccc)$ corresponding to $\boldt$, send it to 
$(\CCC'=(\alpha'_i)_{i=1}^m,\ccc')$ having 
$\alpha'_j = \alpha_j$ and $c'_j=c_j$ for all $j \neq i,i+1$, and
\begin{equation}
\label{braid generator acting on weighted circuits}
\begin{pmatrix}
 \alpha_i &\alpha_{i+1} \\
 & \\
c_i       & c_{i + 1} 
\end{pmatrix} \\
\overset{\sigma_i}{\longmapsto}
\begin{pmatrix}
 \alpha_{i + 1}                                                                    & s_{\alpha_{i + 1}}(\alpha_i) \\
 & \\
 c_{i + 1} + \displaystyle\frac{2\langle\alpha_i, \alpha_{i + 1}\rangle}{|\alpha_{i + 1}|^2} c_i & c_{i} 
\end{pmatrix}.
\end{equation}
Furthermore, the $i$th sign change involution $\epsilon_i$ 
on $(\CCC,\ccc)$ that  replaces 
$\alpha_i \mapsto -\alpha_i$
 and $c_i \mapsto -c_i$ satisfies 
\begin{equation}
\label{sign-change-Hurwitz-interaction}
\begin{aligned}
\sigma_i \epsilon_j &= \epsilon_j \sigma_i &&\text{ for } j \neq i,i+1,\\
\sigma_i \epsilon_i&=\epsilon_{i+1}\sigma_i, &&\text{ and }\\
\sigma_i \epsilon_{i+1}&=\epsilon_i\sigma_i.
\end{aligned}
\end{equation}
\end{proposition}
\begin{proof}
For any root $\alpha$ and any $w$ in $W$
one has
$
s_{\alpha}^{w} = w^{-1} s_{\alpha} w = s_{w^{-1}(\alpha)}
$.  Applying this with $\alpha=\alpha_i$ and $w = s_{\alpha_{i+1}} = w^{-1}$ shows that the pair
$(\CCC',\ccc')$ defined in the statement corresponds to $\boldt'=\sigma_i(\boldt)$. 
The fact that this pair is another $m$-dependence 
% $\sum_{i=1}^m c'_i \alpha'_i=0$ 
comes from $\sum_{i=1}^m c_i \alpha_i=0$ and a calculation with the formula~\eqref{reflection formula}.
% $$
% s_{\alpha_{i+1}}(\alpha_i) = \alpha_i - 
%   \frac{2\langle\alpha_i, \alpha_{i + 1}\rangle}{\langle \alpha_{i + 1},\alpha_{i+1} \rangle} \alpha_{i+1}.
% $$
The inverse $\sigma_i^{-1}:(\CCC,\ccc) \longmapsto (\CCC',\ccc')$
has the following formula:  
$\alpha'_j = \alpha_j$ and $c'_j=c_j$ for all $j \neq i,i+1$, and
\begin{equation}
\label{inverse braid generator acting on weighted circuits}
\begin{pmatrix}
 \alpha_i &\alpha_{i+1} \\
 & \\
c_i       & c_{i + 1} 
\end{pmatrix} \\
\overset{\sigma_i^{-1}}{\longmapsto}
\begin{pmatrix}
 s_{\alpha_i}(\alpha_{i + 1})                                                                    & \alpha_i \\
 & \\
 c_{i+1} & c_{i} + \displaystyle\frac{2\langle\alpha_{i+1}, \alpha_{i}\rangle}{|\alpha_{i}|^2} c_{i+1}  
\end{pmatrix}.
\end{equation}
The relations in \eqref{sign-change-Hurwitz-interaction} are all straightforward
to check.
\end{proof}

\begin{remark}
We will not need it here, but a slightly laborious calculation shows that 
the permutation action of the operators $\sigma_i$ on
the set of $m$-dependences in $\Phi$ satisfies the usual braid relations,
giving an action of the $m$-strand braid group on the set of $m$-dependences.
\end{remark}

%%%%%%%%%%%%%%%%%%%%%
\begin{comment}
\begin{remark}
\label{remark: signs of roots}
There is also a natural action of the group $\{\pm 1\}^m$ on $m$-tuples of roots, where the action is to change signs (or not) of the various entries.  The interaction between this and the Hurwitz action is easy to describe: informally, the braid group permutes the signings of the entries in the only reasonable way.  Formally, let $p \colon \B_m \to \Symm_m$ be the homomorphism from the braid group to the symmetric group sending the braid generator $\sigma_i$ to the simple transposition $s_i = (i \; i + 1)$.  Then for a braid group element $\pi$, a signing $\xx \in \{\pm 1\}^m$, and a tuple $\boldalpha$ of roots, we have
\[
\pi(\xx \cdot \boldalpha) = (p(\pi) \cdot \xx) \cdot (\pi (\boldalpha)),
\]
where $\Symm_m$ acts on $\{\pm 1\}^m$ by permuting coordinates.  As a consequence of this formula, changing the signs of certain roots in a tuple and then applying a braid generator does not change any of the matroidal properties of the resulting tuple of roots.
\end{remark}
\end{comment}
%%%%%%%%%%%%%%%%%%%%%

%%%
\subsubsection{Dealing with issue (ii).}

We begin by studying how the two Hurwitz moves
$\sigma_i, \sigma_i^{-1}$ affect the weight of a dependence.

\begin{proposition}
\label{proposition: obtuse angles are good}
Consider a nontrivial $m$-dependence $(\CCC,\ccc)$ for $m \geq 3$,
with $\CCC=(\alpha_1,\ldots,\alpha_m)$ supported on a circuit $C = \{\alpha_1,\ldots,\alpha_m\}$.
\begin{compactenum}[(a)]
\item
If $\langle c_i \alpha_i, c_{i+1} \alpha_{i+1} \rangle =0$,
then $\wt(\sigma_i(\CCC,\ccc)) = \wt(\sigma_i^{-1}(\CCC,\ccc)) = \wt(\CCC,\ccc)$.
\item
If $\langle c_i \alpha_i, c_{i+1} \alpha_{i+1} \rangle >0$,
then both $\wt(\sigma_i(\CCC,\ccc)),\wt(\sigma_i^{-1}(\CCC,\ccc)) 
> \wt(\CCC,\ccc)$.
\item
If $\langle c_i\alpha_i, c_{i+1}\alpha_{i+1}\rangle <0$,
then either $\wt(\sigma_i(\CCC,\ccc))<\wt(\CCC,\ccc)$
or $\wt(\sigma_i^{-1}(\CCC,\ccc)) <\wt(\CCC,\ccc).$
\end{compactenum}
\end{proposition}

\begin{proof}
Since $C$ is a circuit, all entries of $\ccc$ are nonzero.
Assertion (a) follows because $c_i,c_{i+1}\neq 0$ imply
$\langle \alpha_i, \alpha_{i+1} \rangle=0$, so that $\sigma_i^{\pm 1}$ simply permute the coefficients.

In arguing assertions (b), (c), it is convenient to have all entries $c_j > 0$
in $\ccc$.  One can reduce to this case by applying
sign change operations $\epsilon_i$ that negate some of the $\alpha_j$,
using the relations \eqref{sign-change-Hurwitz-interaction}.

Then from \eqref{braid generator acting on weighted circuits}, one has
\begin{equation}
\label{weight-change}
\wt(\sigma_i (\CCC,\ccc)) - 
\wt(\CCC,\ccc) =c'_i-c_{i+1}
\quad \text{ where } \quad
c'_i: =
\left|c_{i + 1} + 
 \frac{2\langle \alpha_i, \alpha_{i + 1}\rangle}
       { \lvert \alpha_{i + 1} \rvert^2} c_{i}\right|.
\end{equation}

For assertion (b), note that $c_i, c_{i+1}>0$ implies that
$\langle \alpha_i, \alpha_{i + 1}\rangle > 0$, and hence 
$$
c'_i -  c_{i + 1} =
 \frac{2\langle \alpha_i, \alpha_{i + 1}\rangle}
       { \lvert \alpha_{i + 1} \rvert^2} c_{i} 
> 0,  
$$
so that  
$\wt(\sigma_i (\CCC,\ccc)) > \wt(\CCC,\ccc)$.  Moreover the same holds when 
$\sigma_i$ is replaced by $\sigma_{i}^{-1}$, since this only has the effect of switching 
$i$ and $i + 1$ everywhere in \eqref{weight-change}.

For assertion (c), let us assume that
$\langle c_i\alpha_i, c_{i+1}\alpha_{i+1}\rangle <0$ and that both
\begin{equation}
\label{contrapositive-inequality-assumption}
\wt(\sigma_i (\CCC,\ccc)) \geq  \wt(\CCC,\ccc) 
\quad \textrm{ and } \quad 
\wt(\sigma^{-1}_i (\CCC,\ccc)) \geq \wt(\CCC,\ccc),
\end{equation}
in order to reach a contradiction.  Note that
\begin{align*}
\wt(\sigma_i (\CCC,\ccc)) \geq \wt(\CCC,\ccc)
&\quad \Longleftrightarrow \quad
\left|c_{i + 1} + \frac{2\langle \alpha_i, \alpha_{i + 1}\rangle}{|\alpha_{i + 1}|^2} c_{i}\right|
 \geq c_{i + 1}\\
&\quad \Longleftrightarrow \quad 
\left|\frac{2\langle \alpha_i, \alpha_{i + 1}\rangle}{|\alpha_{i + 1}|^2}\right| c_i
 \geq 2 c_{i + 1}
& (\text{since }c_i, c_{i + 1} > 0\text{ and }\langle \alpha_i, \alpha_j \rangle < 0),\\
&\quad \Longleftrightarrow \quad
\left|\frac{\langle \alpha_i, \alpha_{i + 1}\rangle}{|\alpha_{i + 1}|^2}\right| 
 \geq \frac{c_{i + 1}}{c_i}.
\end{align*}
Similarly, swapping the indices $i, i+1$, one has
\[
\wt(\sigma^{-1}_i (\CCC,\ccc)) \geq \wt(\CCC,\ccc)
\quad \Longleftrightarrow \quad 
\left|\frac{\langle \alpha_i, \alpha_{i + 1}\rangle}{|\alpha_{i}|^2}\right| 
\geq
\frac{c_{i}}{c_{i + 1}}.
\]
Therefore the assumption \eqref{contrapositive-inequality-assumption} implies that
\[
\left(\frac{\langle \alpha_i, \alpha_{i+1}\rangle}{|\alpha_i|\cdot|\alpha_{i+1}|}\right)^2
=
\left|\frac{\langle \alpha_i, \alpha_{i + 1}\rangle}{|\alpha_{i + 1}|^2}\right| 
\cdot \left|\frac{\langle \alpha_i, \alpha_{i + 1}\rangle}{|\alpha_{i}|^2}\right| 
\geq
 \frac{c_{i + 1}}{c_i} \cdot
\frac{c_{i}}{c_{i + 1}}
=
1.
\]
Cauchy-Schwarz then forces $\alpha_{i+1}=\pm \alpha_i$,
contradicting $\CCC=(\alpha_1,\ldots,\alpha_m)$ being a circuit with $m \geq 3$.
\end{proof}

With this in hand, issue (ii) is dealt with by the following result.

\begin{proposition}
\label{proposition:acuteness-do-not-obstruct}
Given a nontrivial $m$-dependence $(\CCC,\ccc)$ for $m \geq 3$,
with $\CCC=(\alpha_1,\ldots,\alpha_m)$ 
supported on a full circuit $C = \{\alpha_1,\ldots,\alpha_m\} \subset \Phi$,
there exists another $m$-dependence $(\CCC', \ccc')$ in its Hurwitz orbit
such that $\wt(\ccc') < \wt(\ccc)$.
\end{proposition}
\begin{proof}
As before, $c_i \neq 0$ for all $i$ since $C$ is a circuit.
The acuteness graph $\Gamma_C$ is disconnected by
Lemma~\ref{acuteness-lemma}, so one has a nontrivial decomposition 
$
\{1,2,\ldots,m\}= I \sqcup J
$
in which 
$c_i \alpha_i, c_j \alpha_j$ are nonacute for every $(i,j) \in I \times J$.
They cannot always be {\it orthogonal}, else $C$ would not be a circuit.
Hence, there exists at least one $(i_0,j_0) \in I \times J$
for which $c_{i_0} \alpha_{i_0}, c_{j_0} \alpha_{j_0}$ are (strictly) obtuse.
Assume $i_0 < j_0$ without loss of generality.

Let 
$
I=\{i_1 < i_2 < \cdots\}$ 
and 
$J=\{j_1 < j_2 < \cdots\}
$,
and imagine the process of sorting the sequence $(1,2,3,\ldots,m)$ into
the linear order 
$
(j_1,j_2,\ldots,i_1,i_2,\ldots)
$
using adjacent transpositions $s_k$ (i.e., $s_k$ swaps the entries
in {\it positions} $k,k+1$) so that at each step, the transposed 
values $\{i, j\}$ satisfy $(i,j) \in I \times J$.
Since this process starts with $i_0$ left of $j_0$ and ends with
$i_0$ right of $j_0$, there must exist some {\it first} step in
this process where one uses some $s_{k_0}$ to swap a pair  $(i,j) \in I \times J$
having $c_i \alpha_i, c_j \alpha_j$ obtuse.  All of the previous
steps swap pairs of orthogonal roots, and hence lift
to a corresponding sequence of Hurwitz moves $\sigma_i$
applied to $(\CCC,\ccc)$ that only re-order the entries.  
The product of these moves is a (re-ordered) $m$-dependence $(\CCC'\ccc')$
having $\wt(\ccc')=\wt(\ccc)$.  However,
at the next step, 
Lemma~\ref{proposition: obtuse angles are good}(c) 
shows that one of the two lifts $\sigma^{\pm 1}_{k_0}$
of $s_{k_0}$ will have 
$
\wt( \sigma^{\pm 1}_{k_0}(\CCC',\ccc') ) < \wt(\CCC',\ccc') = \wt(\CCC,\ccc) 
$,
as desired.
\end{proof}

\subsubsection{Dealing with issue (iii)}

We need to know that, after starting with an 
an $m$-dependence and applying a sequence of 
Hurwitz moves that decrease its weight at each step, 
the result cannot be the same $m$-dependence.

\begin{proposition}
\label{repeated-circuits-have-same-weight}
Fix a circuit $C$ in a finite root system $\Phi$ of rank at least $3$, and two $m$-dependences 
$(\CCC,\ccc), (\CCC',\ccc')$ supported on $C$.  If the two $m$-dependences are
in the same Hurwitz orbit then
$
\wt(\ccc')=\wt(\ccc)
$.
\end{proposition}
\begin{proof}
Note that the hypotheses and conclusion of the proposition are
unaffected by rescaling $\ccc$.  

Let $K$ be the finite extension of $\QQ$
generated by 
$\frac{2\langle \alpha,\beta\rangle}{|\alpha|^2}$ for all
roots  $\alpha, \beta$ in $\Phi$.  Every root $\alpha$
is in the $W$-orbit of a root in $\Pi$, and hence by \eqref{reflection formula} 
in the $K$-subspace of $V$ generated by $\Pi$.  Thus, we may 
rescale $\ccc$ so that it lies in $K^m$.  Clearing denominators,
we can assume $\ccc$ lies in $\ooo^m$, 
where $\ooo$ is the ring of integers within $K$.  

We further claim that one can assume both that
$\ooo$ is a {\it principal ideal domain}, and that it contains
all of the algebraic numbers 
$\frac{2\langle \alpha,\beta\rangle}{|\alpha|^2}$ for 
$\alpha, \beta$ in $\Phi$.  To see this claim,
note that our finite root systems $\Phi$ of rank at least $3$ have been chosen
either to be
\begin{compactitem}
\item crystallographic, so that $\ooo=\ZZ \subset \QQ=K$, with
$\frac{2\langle \alpha,\beta\rangle}{|\alpha|^2}$ in $\ZZ$, or
\item type $\type{H}_3, \type{H}_4$, so that
$\ooo=\ZZ\left[(1+\sqrt{5})/2\right] \subset \QQ[\sqrt{5}]=K$,
with
$
\frac{2\langle \alpha,\beta\rangle}{|\alpha|^2} 
\in \left\{2\cos(\frac{2\pi}{5}), 2\cos(\frac{4\pi}{5}) \right\} 
\subset \ooo.
$
\end{compactitem}
Therefore the ideal $I =(\ccc)$ of $\ooo$ generated 
by the entries of $\ccc$ is a \defn{principal
ideal} $I=(g)$ in $\ooo$, where $g:=\gcd(\ccc)$ is uniquely
defined up to scaling by units in $\ooo^\times$.  
The formulas \eqref{braid generator acting on weighted circuits},
\eqref{inverse braid generator acting on weighted circuits}
show that the Hurwitz moves $\sigma^\pm_i$ do not change $I$.

Now assume we are given $(\CCC,\ccc), (\CCC,'\ccc')$ as in the hypothesis of the proposition,
and permute indices of $\CCC$ so that $\CCC'=\CCC$ as (ordered) circuits.
The uniqueness of the dependence up to scaling 
forces $\ccc'= k \ccc$ for some $k \in K^\times$, and hence 
\[
\wt(\CCC',\ccc')=|k| \cdot \wt(\CCC,\ccc).
\]
The above discusion shows that, additionally,
$\gcd(\ccc)=\gcd(\ccc')$ in $\ooo$, so that  $\ccc'= k \ccc$
forces $k$ to lie in $\ooo^\times$.
Thus, in the crystallographic case, we are done since
$\ooo^\times=\ZZ^\times=\{\pm 1\}$, so $|k|=1$.

In the noncrystallographic  $\type{H}_3, \type{H}_4$ cases,
we still need to rule out the possibility that the unit
$k$ in $\ooo^\times$ has $|k| \neq 1$.  This would
mean that the Hurwitz orbit of the $m$-dependence 
$(\CCC,\ccc)$ contains {\it infinitely} many other elements, namely
those whose weights are scaled by  $1, |k|, |k|^2,\ldots.$  However, we used
a computer to check that this does not happen:
for every $W$-orbit of full circuits $C$ in 
$\Phi_{\type{H}_3}, \Phi_{\type{H}_4}$, as classified in 
Section~\ref{exceptional-type-circuit-section}, we linearly ordered $C$
in all ways to form $\CCC$, picked coefficients
$\ccc$ (uniquely up to scaling) 
to create an $m$-dependence $(\CCC,\ccc)$, applied
all Hurwitz moves $\sigma^\pm_i$ to generate new dependences,
then repeated with the new dependences.
\emph{A priori} this could have run
indefinitely, but in fact it always
terminated with a finite list, proving the claim.
\end{proof}

\subsubsection{Proof of Lemma~\ref{main lemma} in ranks at least $3$.}

Given a reflection factorization $\boldt = (t_1, \ldots, t_m)$ of
$w = t_1 \cdots t_m$ with $\ell_T(w) < m$ and $m \geq 3$,
we want to show there exists $\boldt' = (t'_1, \ldots, t'_m)$ 
in the Hurwitz orbit of $\boldt$ with 
$\ell_T(t'_1 \cdots t'_k) < k$ for some $k \leq m-1$.

As mentioned earlier, 
using Propositions~\ref{Carter's lemma} and \ref{prefix-prop}, we may assume without loss of generality that
the tuple $\CCC=(\alpha_1,\alpha_2,\ldots,\alpha_m)$ of roots corresponding to 
$(t_1, \ldots, t_m)$ via $t_i = s_{\alpha_i}$ is supported on a circuit
$C=\{\alpha_1,\ldots,\alpha_m\}$ in $\Phi$.
Furthermore, as in Section~\ref{section: acuteness},
one can also assume that $\Phi_W$ is irreducible,
and that $C$ is a full circuit in $\Phi_W$.

Pick coefficients $\ccc$ that make $(\CCC,\ccc)(=:(\CCC^{(0)},\ccc^{(0)}))$ 
an $m$-dependence.  Then 
Proposition~\ref{proposition:acuteness-do-not-obstruct} shows that
that there exists an $m$-dependence
$(\CCC^{(1)},\ccc^{(1)})$ in its Hurwitz orbit having
$\wt(\CCC^{(1)},\ccc^{(1)}) < \wt(\CCC^{(0)},\ccc^{(0)})$.  Repeat this process, producing
a sequence of $m$-dependences $(\CCC^{(i)},\ccc^{(i)})$ in
the Hurwitz orbit of $(\CCC,\ccc)$, with {\it strictly decreasing} sequence
of weights. If it ever happens that some coefficient $c^{(i)}_j=0$,
so that some proper subsequence of $\CCC^{(i)}$ is dependent,
then we are done by Propositions~\ref{Carter's lemma} and \ref{prefix-prop}.
However, this {\it must} happen:  otherwise
each $\CCC^{(i)}$ is supported on a circuit 
$C^{(i)} \subset \Phi$, of which there are only finitely many,
so $C^{(i)}=C^{(j)}$ for some $i < j$, contradicting 
Proposition~\ref{repeated-circuits-have-same-weight}.

This completes the proof of Lemma~\ref{main lemma} in rank at least three,
and hence in all ranks.

%%%%%%%%%%%%
\subsection{Proof of Corollary~\ref{cor:standard form}}

Recall the statement of Corollary~\ref{cor:standard form} from the Introduction.

\vskip.1in
\noindent
{\bf Corollary~\ref{cor:standard form}.}
{\it
If $\ell_T(w)=\ell$, then every factorization of $w$ into $m$ reflections
lies in the Hurwitz orbit of some $\boldt=(t_1, \ldots, t_m)$ such that
\begin{align*}
t_1 &= t_2,\\
t_3 &= t_4,\\
    &\vdots \\
t_{m - \ell - 1} &= t_{m - \ell},
\end{align*} 
and
$(t_{m - \ell + 1}, \ldots, t_{m})$ is a shortest
reflection factorization of $w$.
}
\vskip.1in

\begin{proof}
Induct on $m$, with trivial base case $m=0$.  In the inductive step for $m > 0$, given a
reflection factorization $w=t_1 t_2 \cdots t_m$, either $\ell:=\ell_T(w)=m$,
in which case we are done, or there exists some smallest index $i$ 
for which $\ell_T(t_1 t_2 \cdots t_i) < i$.  By applying
Lemma~\ref{main lemma} repeatedly, we may assume that $i=2$.
This means that $t_1=t_2$, and we are done by applying the induction to the factorization $w = t_3 t_4 \cdots t_m$.  
\end{proof}

%%%%%%%%%%%%%%%%%%%%%%%%%%%%%%%%%%%%%%%%%%%%%%%%%%%%%%%%%%%%%
\section{Coxeter elements and the proof of Theorem~\ref{main theorem}}
\label{main theorem section}
%%%%%%%%%%%%%%%%%%%%%%%%%%%%%%%%%%%%%%%%%%%%%%%%%%%%%%%%%%%%%

Theorem~\ref{main theorem} is a statement about factorizations of
Coxeter elements.  We recall their definition and a few properties here,
before proving the theorem.

\begin{definition}
\label{Coxeter-element-definition}
Given a finite real reflection group $W$ with root system $\Phi$,
one defines a \defn{Coxeter element} to be any element of $W$ of the
form $c=s_{\alpha_1} s_{\alpha_2} \cdots s_{\alpha_n}$, where $(\alpha_1,\ldots,\alpha_n)$ is
any ordering of any choice of simple roots $\Pi=\{\alpha_1,\ldots,\alpha_n\}$ for $\Phi$.
\end{definition}

It turns out (see, e.g. \cite[\S 3.16]{Humphreys}) that all 
Coxeter elements $c$ lie within the same $W$-conjugacy class.
We mention here a few other important properties of Coxeter elements
that we will use.  One is Bessis's Theorem \cite[Prop. 1.6.1]{Bessis}
from the Introduction, asserting
that any two shortest reflection factorizations $c=t_1 t_2 \ldots t_n$ 
lie in the same Hurwitz orbit.  It has the following non-obvious corollary.

\begin{corollary}
\label{nonobvious-Bessis-multiset-corollary}
In a finite real reflection group, any two shortest reflection factorizations
of a Coxeter element use the same multiset of reflection conjugacy classes.
\end{corollary}

\noindent
(Specifically, it is the multiset of conjugacy classes of the simple root reflections $(s_{\alpha})_{\alpha \in \Pi}$, two of which lie in the same $W$-conjugacy class if and only if they have a path of odd-labeled edges between them 
in the {\it Coxeter diagram} for $W$;  see \cite[Exer.~1.16]{BjornerBrenti}.)

We will also need the following lemma used by Bessis in the proof of his theorem.

\begin{lemma}[{Bessis \cite[Lem.~1.4.2]{Bessis}}]
\label{reflections-are-noncrossing}
For every Coxeter element $c$ and reflection $t$ in $W$ 
there exists a shortest reflection factorization $c=t_1 t_2 \cdots t_n$ starting with $t_1=t$.
\end{lemma}

Combining Bessis's Theorem from the Introduction with Lemma~\ref{reflections-are-noncrossing} gives the following.

\begin{corollary}
\label{cor:Hurwitz-to-get-t-at-beginning}
Fix a reflection $t$ and a Coxeter element $c$.  Then 
every shortest reflection factorization $c=t_1 t_2 \cdots t_n$
lies in the Hurwitz orbit of such a factorization that starts with $t$.
\end{corollary}

We can now prove Theorem~\ref{main theorem} from the Introduction, whose statement we recall here.

\vskip.1in
\noindent
{\bf Theorem~\ref{main theorem}.}
{\it
In a finite real reflection group, two reflection factorizations 
of a Coxeter element lie in the same Hurwitz orbit 
if and only if they share the same multiset of conjugacy classes.
}

\begin{proof}
The ``only if" statement is clear, as Hurwitz moves do not affect the multiset of
conjugacy classes.

For the ``if" statement, given two reflection factorizations
 $\boldt = (t_1, \ldots, t_m)$ and $\boldt' = (t'_1, \ldots, t'_m)$ of $c$ 
having the same multiset of reflection conjugacy classes, we show that they
lie in the same Hurwitz orbit via induction on $m$.
By Corollary~\ref{cor:standard form}, we may assume that both $\boldt, \boldt'$ consist of a sequence of $(m - n)/2$ pairs of equal reflections, followed by shortest factorizations $\hat{\boldt}, \hat{\boldt}'$ of $c$:
\begin{align*}
%\boldt &= (t_1, t_1,\;  t_3, t_3, \; \ldots,  \; t_{m - n - 1}, t_{m - n - 1},\;\, \quad t_{m - n + 1}, t_{m - n + 2}, \ldots, t_m)\\
%\boldt'&=(t'_1, t'_1,\; t'_3, t'_3,\;  \ldots,  \;  t'_{m - n - 1}, t'_{m - n - 1}, ,\;\, \quad t'_{m - n + 1}, t'_{m - n + 2}, \ldots, t'_m).
\boldt &= 
(t_1, t_1,\;  t_3, t_3, \; \ldots,  \; t_{m - n - 1}, t_{m - n - 1},\;\, \hat{\boldt} ), \\
\boldt'&=(t'_1, t'_1,\; t'_3, t'_3,\;  \ldots,  \;  t'_{m - n - 1}, t'_{m - n - 1}, \;\,  \hat{\boldt}' ).
\end{align*}
It would suffice to show that the Hurwitz orbit of $\boldt$ contains a factorization that
starts with $(t'_1,t'_1)$, since one could then apply induction after restricting
$\boldt, \boldt'$ to their last $m-2$ positions $\{3,4,\ldots,m\}$.

To this end, we first claim that one of the pairs $(t_i, t_i)$ (in positions $i, i + 1$) of adjacent equal reflections in $\boldt$ has $t_i$ in the same conjugacy class as $t'_1$;  this is so because
Corollary~\ref{nonobvious-Bessis-multiset-corollary} implies $\hat{\boldt}, \hat{\boldt}'$
share the same multiset of conjugacy classes, and it is a hypothesis of the theorem that $\boldt, \boldt'$ share the same multiset of conjugacy classes.

Via a sequence of Hurwitz moves of the form $\sigma_k^{-1}$, one can move
%
%\[
%(\sigma_{m - \ell - 2}^{-1} \cdots \sigma_{i + 1}^{-1} \sigma_{i}^{-1})
%\cdot
%(\sigma_{m - \ell - 1}^{-1} \cdots \sigma_{i + 2}^{-1} \sigma_{i + 1}^{-1})
%\]
the two copies $(t_i,t_i)$ in $\boldt$ to the right, stopping just before $\hat{\boldt}$, giving
an element in the Hurwitz orbit of $\boldt$ whose last $n+2$ positions are
\[
(t_i, t_i, \hat{\boldt}).
\]
%Use the notation $\sim$ to denote the equivalence relation ``belonging to the same Hurwitz orbit''.
Since $t_i$ is $W$-conjugate to $t_1'$, one can choose $w$ in  $W$ 
and a reflection factorization $w = r_1 \cdots r_k$ such that 
\[
t'_1 
\quad = \quad 
w^{-1}t_iw 
\quad = \quad 
t_i^{w}
\quad = \quad 
t_i^{r_1 r_2 \cdots r_k}.
\]
By Corollary~\ref{cor:Hurwitz-to-get-t-at-beginning}, one can apply Hurwitz moves
to $\hat{\boldt}$ and make it start with the reflection $r_1$.  Thus the last $n+2$ positions
in the factorization now look like 
\[
(t_i, t_i, r_1, \hat{\hat{\boldt}}).
\]
%\[
%\begin{array}{cccccccc}
%\boldt & \sim & ( &\ldots, & t_i, t_i, &                    & \widetilde{\boldt}&) \\
%      & =    & ( &\ldots, & t_i, t_i, &        u_1,        &  \ldots           &) \\
%     & \sim & ( &\ldots, &  r_1,     & t_i^{r_1}, t_i^{r_1},&  \ldots           &) \\
%       & \sim & ( &\ldots, & t_i^{r_1}, t_i^{r_1}, &          & \widetilde{\boldt}&)
%\end{array}
%\]
Apply the Hurwitz moves of the form $\sigma_k$ that move $r_1$ two steps left, changing the factorization to 
\[
(r_1, t_i^{r_1}, t_i^{r_1}, \hat{\hat{\boldt}}).
\]
Then apply Hurwitz moves of the form $\sigma_k$ that move both of $(t_i^{r_1}, t_i^{r_1})$
one step to the left, changing it to
\[
(t_i^{r_1}, t_i^{r_1},r_1^{t_i^{r_1} t_i^{r_1}},  \hat{\hat{\boldt}})
\quad = \quad
(t_i^{r_1}, t_i^{r_1}, r_1,  \hat{\hat{\boldt}}),
\]
%(where the second equivalence is by $\sigma_{m - \ell- 1}\sigma_{m - \ell}$ and the third is by $\sigma_{m - \ell}\sigma_{m - \ell - 1}$).
where the suffix $(r_1,  \hat{\hat{\boldt}})$ is still a shortest factorization of $c$.
Repeating this process with $r_2, r_3, \ldots, r_k$ in place of $r_1$ gives a factorization
whose last $n+2$ positions have the form
$$
( t_1', t_1',\tilde{\boldt})
$$
%\[
%\boldt \sim (\ldots, t_i^w, t_i^w, \ldots) =  (\ldots, t'_1, t'_1, \ldots),
%\]
for some shortest factorization $\tilde{\boldt}$ of $c$.
Then applying a sequence of moves of the form $\sigma_k$ gives a factorization in the Hurwitz orbit of $\boldt$
that moves $(t_1',t_1')$ to the first two positions, as desired. 
\end{proof}

%%%%%%%%%%%%%%%%%%%%%%%%%%%%%%%%%%%%%
\section{Remarks}
\label{remarks section}
%%%%%%%%%%%%%%%%%%%%%%%%%%%%%%%%%%%%%

\subsection{Quasi-Coxeter elements}
\label{Wegener-section}

Baumeister, Gobet, Roberts, and Wegener \cite{BGRW} define a 
{\it quasi-Coxeter element} $c$ in a finite reflection group $W$ to
be an element having a shortest reflection
factorization $c=t_1 t_2 \cdots t_\ell$ for which $\{t_1,t_2, \ldots,t_\ell\}$
generates $W$.  For example, Coxeter elements as in
Definition~\ref{Coxeter-element-definition} have this property.
P. Wegener has pointed out that our proof of Theorem~\ref{main theorem} 
generalizes to prove the following.

\begin{theorem}
\label{quasi-Coxeter-theorem}
In a finite real reflection group, two reflection factorizations 
of a quasi-Coxeter element lie in the same Hurwitz orbit 
if and only if they share the same multiset of conjugacy classes.
\end{theorem}

\begin{proof}[Proof sketch.]
The crucial Corollary~\ref{cor:standard form} applies to {\it any} element 
of $W$.  Also, Bessis's Theorem from the Introduction, asserting transitivity
of the Hurwitz action for shortest reflection factorization of Coxeter elements,
was generalized to quasi-Coxeter elements as \cite[Thm.~1.1]{BGRW}.
This implies the analogue of 
Corollary~\ref{nonobvious-Bessis-multiset-corollary},
replacing the word ``Coxeter element'' by ``quasi-Coxeter element.''
In place of of Lemma~\ref{reflections-are-noncrossing}
one uses the following property of a 
quasi-Coxeter element $c$: 
$W$ is generated by the set of $t_1$ 
that can appear in the {\it first position}
of a shortest reflection factorization $c=t_1 t_2 \cdots t_\ell$.
This is because there is a shortest factorization $c=t'_1 t'_2 \cdots t'_\ell$ 
for which $\{ t'_1,t'_2,\ldots,t'_\ell \}$ generates $W$, and  
Proposition~\ref{prefix-prop} implies that $ t'_1,t'_2,\ldots,t'_\ell $ all appear first in some shortest factorization of $c$.
The rest of the proof of
Theorem~\ref{main theorem} uses only these properties.
\end{proof}

In fact, the quasi-Coxeter property seems to go to the heart of Hurwitz transitivity for factorizations of arbitrary length.  
For example, in a Coxeter group $W$ having only one reflection conjugacy class, 
if one is given a non-quasi-Coxeter element $w$, one can choose 
a reflection factorization $w=t_1 t_2 \cdots t_m$ such that 
$W':=\langle t_1, \ldots, t_m \rangle \subsetneq W$.
Then for any reflection $t$ in $W \setminus W'$, the two 
factorizations 
$
w=t_1 \cdots t_m \cdot t \cdot t 
=t_1 \cdots t_m \cdot t_1 \cdot t_1
$
of length $m + 2$ use the same multiset of reflection conjugacy classes, but
necessarily lie in different Hurwitz orbits.

\subsection{Affine Weyl groups}

Note that the crucial Lemma~\ref{main lemma}
holds for the smallest case of an {\it affine Weyl group},
namely, the infinite dihedral group $W_\infty$ of type $\type{I}_2(\infty)$
from Section~\ref{rank-two-standard-form-section}.  It is also not
hard to check that Theorem~\ref{main theorem} holds verbatim for this
group $W_\infty$, raising the following question.

\begin{question}
Does Theorem~\ref{main theorem} hold verbatim for affine Weyl groups?
Other non-finite Coxeter groups?
\end{question}

Note that there is an issue in the definition of {\it Coxeter elements} for arbitrary 
Coxeter systems $(W,S)$, since products $c=s_1 s_2 \cdots s_n$ of the 
elements of $S$ in different orders need not be 
$W$-conjugate if the Coxeter diagram contains cycles.

%We would probably want to mention
%the paper by McCammond and Petersen about absolute length,
%which ought to be relevant for this, and the {\it volume, parallelism}
%reduction that might be suggested by the proof of 
%Lemma~\ref{main lemma} in rank $2$ that appears in 
%Section~\ref{rank-two-standard-form-section}.

\subsection{Complex reflection groups}

As mentioned in the Introduction, Bessis extended 
his theorem on shortest factorizations 
from real reflection groups to {\it well-generated complex 
reflection groups}, where the notion of {\it Coxeter elements}
still makes sense; see \cite{BessisKpi1}.
In fact, all evidence points to the following verbatim generalization of
Theorem~\ref{main theorem}.

\begin{conjecture}
\label{well-generated-conjecture}
In a well-generated finite 
complex reflection group, two reflection factorizations 
of a Coxeter element lie in the same Hurwitz orbit 
if and only if they share the same multiset of conjugacy classes.
\end{conjecture}

We discuss some of the evidence for 
Conjecture~\ref{well-generated-conjecture} here.
Just as with real reflection groups, there is
a classification of all finite complex reflection groups acting
irreducibly, due to Shephard and Todd. It contains one infinite
family $G(de,e,n)$ for $n,d,e \geq 1$, and $34$ exceptional groups.
The group $G(de,e,n)$ consists of all $n \times n$ matrices
which are {\it monomial} (that is, having exactly one nonzero entry
in each row and column) and whose nonzero entries are
$de$th roots of unity, with their 
product a $d$th root of unity.\footnote{This contains as
special cases the real types $\type{A}_{n-1}$ as $G(1,1,n)$
(restricted to act on the hyperplane $(1,1,\ldots,1)^\perp$),
types $\type{B}_n/\type{C}_n$ as $G(2,1,n)$, 
type $\type{D}_n$ as $G(2,2,n)$, and type $\type{I}_2(m)$
as $G(m,m,2)$.}

Although every real reflection group is {\it well-generated}
in the sense of having a generating set consisting of $n$ reflections,
this is not true for all complex reflection groups $W$ acting on $\CC^n$.
For example, within the infinite family $G(de,e,n)$, this fails when 
$d,e,n \geq 2$; only the subfamilies $G(d,1,n), G(e,e,n)$ are
well-generated.

The first author has verified Conjecture~\ref{well-generated-conjecture}
via a direct argument in $G(d,1,n)$.  We have verified
it via computer for the factorizations $c=t_1 t_2 \cdots t_m$ with $m \leq n+3$
in the following well-generated groups acting irreducibly on $\CC^n$:
$G(e,e,n)$ with $(n,e)=(3, 3), (3, 4), (3, 5), (3, 6), (4, 3)$ and 
Shephard-Todd's exceptional
types $\type{G}_4,\type{G}_5,\type{G}_6,\type{G}_8$.  

Regarding proof techniques, one might hope that 
Lemma~\ref{main lemma} generalizes to all well-generated groups.
%
%
% Unfortunately, this is not the case.  For example, consider the exceptional group $W$ of type $\type{G}_5$, 
% with presentation $W \equiv \langle s, t \mid s^3 = t^3 = e, stst = tsts \rangle$.  The reflections in $W$
% are the conjugates of $s$, $t$, $s^{-1}$ and $t^{-1}$.  The element $w := sts^{-1}ts = t^{-1} \cdot st^{-1}sts^{-1}$ has reflection length $2$, 
% and has $60$ non-minimal factorizations as a product of three reflections.  These $60$ factorizations fall into four Hurwitz orbits, one of which includes the factorization
% \[
% w = s \cdot t \cdot s^{-1}ts.
% \]
% One may check that all $27$ factorizations $(t_1,t_2,t_3)$
% in the Hurwitz orbit of this factorization 
% $(s, t ,s^{-1}ts)$ share the property that the product $t_1 t_2$ has
% reflection length $2$, and so Lemma~\ref{main lemma} does not hold.  
% It is not clear what might replace this
% lemma in a proof of Conjecture~\ref{well-generated-conjecture}.
%
%
Unfortunately, this is not the case, even in the infinite family $G(d, 1, n)$.  For example, consider $W := G(d, 1, 2)$ with $d > 2$, and let $\zeta$ be a primitive $d$th root of unity.  The length-$3$ factorization
\[
\begin{array}{ccccc}
w &=&t_1 &t_2&t_3\\[.1in]
\begin{bmatrix}
\zeta^2 & 0 \\
0 & \zeta^{-1}
\end{bmatrix}
&=&
\begin{bmatrix}
0 & 1 \\
1 & 0
\end{bmatrix}
&
\begin{bmatrix}
0 & \zeta^{-1} \\
\zeta & 0
\end{bmatrix}
&
\begin{bmatrix}
\zeta & 0 \\
0 & 1
\end{bmatrix}
\end{array}
\]
in $W$ is not shortest, as $\ell_T(w)=2$.  
However, one can check that for any 
$(t_1',t_2', t_3')$ within the Hurwitz orbit of $(t_1,t_2,t_3)$, 
the prefix $(t_1',t_2')$ is a shortest 
factorization of $t_1' t_2'$. 
%\footnote{Indeed, the product of any diagonal reflection and any ``transposition-like'' reflection in $W$ has reflection length $2$; the product of two transposition-like reflections may have either reflection length $2$ or $0$, but in a prefix of a length-$3$ factorization of $w$, length $0$ is not an option as $w$ is not itself a reflection.}. It is not clear what might replace Lemma~\ref{main lemma} in a proof of Conjecture~\ref{well-generated-conjecture}.
It is not clear what might replace Lemma~\ref{main lemma} in a proof of Conjecture~\ref{well-generated-conjecture}.

\bibliography{Hurwitz}{}
\bibliographystyle{alpha}

\end{document}